\documentclass[12pt]{amsart}
\usepackage{amsgen, amsmath, amsfonts, amsthm, amssymb, enumerate,amscd, mathtools}
\usepackage{hyperref}
\usepackage[all]{xy}
\usepackage[dvipsnames]{xcolor}
\usepackage[normalem]{ulem}

\def\ta{\tilde{a}}

\newtheorem{defn}{Definition}[section]
\newtheorem{prop}[defn]{Proposition}
\newtheorem{lem}[defn]{Lemma}
\newtheorem{thm}[defn]{Theorem}
\newtheorem{cor}[defn]{Corollary}
\newtheorem{con}[defn]{Conjecture}

\numberwithin{equation}{section}

\newcommand {\ZZ}{{\mathbb Z}}
\newcommand {\NN}{{\mathbb{N}}}

\newcommand {\C}{{\mathbb C}}

\newcommand {\R}{{\mathbb R}}

\newcommand {\m}{{\mathfrak m}}

\def\gcd{\operatorname{}}

\def\mod{\operatorname{mod}}

\def\Im{\operatorname{Im}}

\newcommand{\bH}{\mathbb{H}}

\DeclareMathOperator{\GL}{GL}

\DeclareMathOperator{\ord}{ord}

\DeclareMathOperator{\cond}{cond}

\newcommand{\sm}{\left(\begin{smallmatrix}}
\newcommand{\esm}{\end{smallmatrix}\right)}
\newcommand{\bpm}{\begin{pmatrix}}
\newcommand{\ebpm}{\end{pmatrix}}

\definecolor{blue}{rgb}{0,0,1}
\definecolor{red}{rgb}{1,0,0}
\definecolor{green}{rgb}{0,.6,.2}
\definecolor{purple}{rgb}{1,0,1}

\long\def\red#1\endred{\textcolor{red}{#1}}
\long\def\blue#1\endblue{\textcolor{blue}{#1}}
\long\def\purple#1\endpurple{\textcolor{purple}{ #1}}
\long\def\green#1\endgreen{\textcolor{green}{#1}}

\title{Additive twists and a conjecture by Mazur, Rubin and Stein}

\medskip

\author{Nikolaos Diamantis}
\email{\tt nikolaos.diamantis@nottingham.ac.uk } 
\address{School of Mathematical Sciences \\ University of Nottingham \\ Nottingham\\ NG7 2RD\\ United Kingdom}

\author{Jeffrey Hoffstein}
\email{\tt jhoff@math.brown.edu}
\address{Mathematics Department
\\
Brown University \\
Providence \\ RI 02912 \\USA}

\author{Eren Mehmet K{\i}ral}
\email{\tt erenmehmetkiral@protonmail.com}
\address{Department of Mathematics\\
Keio University\\
Building 14, 443\\ 3-14-1 Kouhoku-ku, Hiyoshi\\
Yokohama, 223-8522 Japan}

\author{Min Lee}
\email{\tt min.lee@bristol.ac.uk}
\address{School of Mathematics\\University of Bristol\\
University Walk\\Bristol\\BS8 1TW\\United Kingdom}

\date{}
\begin{document}
\maketitle
\begin{abstract}
In this paper, a conjecture of Mazur, Rubin and Stein concerning certain averages of modular symbols is proved. To cover levels that are important for elliptic curves, namely those that are not square-free, we establish results about $L$-functions with additive twists that are of independent interest.
\end{abstract}

\tableofcontents

\section{Introduction}
Motivated by a question regarding ranks of elliptic curves defined over cyclic extensions of $\mathbb Q$, B. Mazur and K. Rubin \cite{MR} studied the statistical behaviour of 
modular symbols associated to a weight $2$ cusp form corresponding to an elliptic curve.  Based on both theoretical and computational arguments (the latter jointly with W. Stein)
they formulated a number of precise conjectures. We state one of them in its formulation given in \cite{PR}. 

For a positive $q$, let $\Gamma_0(q)$ denote the group of matrices
$\left ( \smallmatrix a & b \\ c & d \endsmallmatrix \right )$
of determinant $1$ with $a, b, c, d \in \mathbb Z$ and $q\mid c$.
Let 
\begin{equation}\label{e:f}
f(z)=\sum_{n=1}^{\infty}a(n) e^{2 \pi i n z} 
=\sum_{n=1}^{\infty}A(n)n^{1/2} e^{2 \pi i n z}
\end{equation}
be a newform of weight $2$ for $\Gamma_0(q)$. 

For each $r \in \mathbb Q$, we set
\[\left< r \right>^+=2 \pi  \int_{i \infty}^r \Re(if(z)dz) \qquad \text{and} \qquad 
\left< r \right>^-=2 \pi i \int_{i \infty}^r \Re(f(z)dz).\]
For each $x \in [0, 1]$ and $M \in \mathbb N$, set 
\[G^{\pm}_M(x)
=\frac{1}{M}\sum_{0 \le a \le Mx}\left< \frac{a}{M} \right>^{\pm}.\]
Mazur, Rubin and Stein, in \cite{MR}, stated the following conjecture:
\begin{con} \label{Conj}  For each $x \in [0, 1],$ we have
\begin{align}
\lim_{M \to \infty} G_M^+(x)&= \frac{1}{ 2\pi} \sum_{n \ge 1} \frac{a(n)\sin(2 \pi nx)}{n^2}; \label{e:lim_GM+}\\
\lim_{M \to \infty} G_M^-(x)&= \frac{1}{ 2\pi i} \sum_{n \ge 1} \frac{a(n)(\cos(2 \pi nx)-1)}{n^2}. \nonumber
\end{align}
\end{con}

The heuristic for this conjecture can be seen by the computation
$$ G_M^+(x) 
= \frac{1}{M} \sum_{0\leq a \leq Mx} \left< \frac aM \right>^+
= 2\pi  \Re\left( i \int_\infty^0 \frac{1}{M} \sum_{0 \leq a  \leq Mx}  f\left(\frac{a}{M} + iy\right) i \; dy\right). $$
The inner sum is a Riemann sum for the horizontal integral $\int_0^x$. As a heuristic let us replace the sum with the integral, even though the error is not controlled for small $y$. 
Upon doing this, computing the integral using the Fourier expansion of $f$ gives us the right hand side of \eqref{e:lim_GM+}. 

An average version of this conjecture in the case of square-free levels was proved in \cite{PR}. The same paper contains the proofs of other conjectures  from the original set listed in \cite{MR}. 
More recently, 
one of the original conjectures of \cite{MR}, namely the one dealing with the variance of the modular symbols, was proved in \cite{B..}. 
The authors also established a form of Conjecture~\ref{Conj} in the special case that $x=1$ and $M$ goes to infinity
over a sequence of primes. 

In this paper, we prove Conjecture~\ref{Conj} for an arbitrary level $q$, each $x\in [0, 1]$ and as $M$ goes to infinity over any sequence of integers.
Our main theorem is as follows. 
\begin{thm}\label{MAIN} For each $x \in [0, 1]$, as $M$ goes to infinity, we have
\begin{align*}
G_M^+(x)
&= \frac{1}{ 2\pi } \sum_{n \ge 1} \frac{a(n)\sin(2 \pi nx)}{n^2}
+ 
\mathcal{O}_{\epsilon}\bigg((Mq)^{\epsilon} M^{-\frac{1}{4}} q^{\frac{1}{4}} \prod_{\substack{p\mid \gcd(q, M) \\ p^2\mid q}} p^{\frac{1}{2}}\bigg); 
\\
G_M^-(x)
&= \frac{1}{ 2\pi i} \sum_{n \ge 1} \frac{a(n)(\cos(2 \pi nx)-1)}{n^2}
+
\mathcal{O}_{\epsilon}\bigg((Mq)^{\epsilon} M^{-\frac{1}{4}} q^{\frac{1}{4}} \prod_{\substack{p\mid \gcd(q, M) \\ p^2\mid q}} p^{\frac{1}{2}}\bigg), 
\end{align*}
for any $\epsilon>0$.
\end{thm}

Very recently, H.-S. Sun \cite{KS} announced a similar statement in the special case of $q$ square-free, with a slightly weaker exponent in $q$. The main reason for the difference in our results is that we develop an approximate functional equation for the additive twists of $L$-functions applicable to all levels and additive conductors and that we are then able to solve the difficult problem of bounding the Fourier coefficients of the ``contragredient" function (Proposition~\ref{prop:Fourier_upper})

The starting point of our method was the use of {\it Eisenstein series with modular symbols} in \cite{PR} combined with 
the computation of its Fourier coefficients in terms of shifted convolution series in \cite{BD}. In this paper we succeed in avoiding its use and this simplifies our argument. In an earlier version of the paper, the shifted convolution series itself remained a key tool, but we are now able to circumvent those too. (In this respect, our method parallels that of \cite{KS}).
However, the part we no longer require for the proof of our main theorem contains several methods and results of independent interest and novelty, including {\it double} shifted convolution series. It is one of the themes of work in progress. 

As noted above, previous progress towards the Mazur, Rubin and Stein conjecture concerned only the case of square-free  level (or prime $M$).  Extending to non-square-free levels proved much less routine than we expected and it led to results of independent interest. For example, in Proposition~\ref{lemmodsym} we prove a general bound for antiderivatives of weight $2$ newforms 
\[\int_\infty^{\frac{a}{d}} f(z) \; dz\]
that holds for all rational values $a/d$ and levels $q$.
The proof of this bound is based on another result of independence importance, namely Proposition~\ref{prop:Fourier_upper}. 
As mentioned in \cite[Section~14.9]{IK}, 
the Ramanujan-Petersson bound for Fourier coefficients of a Dirichlet twist of $f$ holds even when the twist is not a newform, but there is an implied constant which may depend on the level badly. 
In Proposition~\ref{prop:Fourier_upper} we make that dependence entirely explicit. 

The twisted cusp form that is the subject of Proposition~\ref{prop:Fourier_upper} appears as the ``contragredient" function in a functional equation for additive twists of $L$-functions for general levels and weights (Theorem~\ref{thm:additive_fe}). This is another result of general applicability which does not seem to appear in the literature in that level of generality and in this explicit form. (However, see the recent preprints \cite{C}, \cite{AC}, which came out during the refereeing process of this paper,  where general Voronoi summation formulas are established, but not in the form required for our purposes).

\subsection{Outline of the proof of Theorem~\ref{MAIN}} 
The  technical aspects of the proof of Theorem~\ref{MAIN} are quite complex, and for that reason we supply here a high level roadmap that we hope will make our proof a bit easier to understand.

The first step in the proof is to express $G_M^{\pm}(x)$ as a sum of modular symbols weighted by a family of smooth functions $h_{\delta}$ that approximates the characteristic function of $[0, x]$. 
This is done in Section~\ref{test}, where an explicit family $\{h_{\delta}\}$ is constructed.
In Lemma~\ref{weights}  it is shown that for any fixed $\delta=\delta_{M} <1$
we have
\[G^{\pm}_M(x)
=\frac{1}{M}\sum_{0 \le a \le M} \left< \frac{a}{M}  \right>^{\pm} h_{\delta}\left(\frac{a}{M}\right) 
+\text{error, uniform in $q$ and $M$}.\]
In view of this expression, in the following sections we focus on sums of the form 
\[\frac{1}{M}\sum_{0 \le a \le M} \left< \frac{a}{M}  \right>^{\pm} h\left(\frac{a}{M}\right)\] 
for an arbitrary smooth period function $h$ on $\mathbb R.$
We have 
\begin{equation}\label{preFINALFORMULA}
\frac{1}{M}\sum_{0 \le a \le M} \left< \frac{a}{M} \right>^{\pm} h\left(\frac{a}{M}\right) 
=\text{an explicit series involving $L(1, f, a/M)$}
\end{equation}
(see \eqref{e:ApmhM} and \eqref{alpha2twist}).

To study the asymptotics of $L(1, f, a/M)$, which is required for the completion of the proof of Theorem \ref{MAIN}, we need a functional equation for $L(s, f, a/M)$ applying to all levels $q$ and integers $M$. 
Since we could not find such a functional equation in the literature in a sufficiently explicit form, we establish it in Section \ref{addtwi}.   
The functional equation and the explicit Ramanujan-type bound for the Fourier coefficients of the twisted cusp forms in the case we need them is the content of Corollary \ref{cFE}.

Two important implications of the functional equation (also of independent interest) are the bound \eqref{modsym} for modular symbols and the approximate functional equation \eqref{approx}, both of which apply to 
arbitrary levels.

In Section \ref{as} we substitute $L(1, f, a/M)$ in the right hand side of \eqref{preFINALFORMULA}, using the approximate functional equation \eqref{approx}, 
 and this leads us to an expression \eqref{a(1)} consisting of two parts. 

The first part is shown (in Lemma~\ref{asym}) to contribute the main term.  
The second part is complicated, but can be bounded using Weil's bound for Kloosterman sums 
and the explicit Ramanujan-type bound for the Fourier coefficients of twisted cusp forms proved in Corollary~\ref{cFE}. 
Combining these two pieces together, we deduce
\begin{multline}\label{pretemp}
\frac{1}{M}\sum_{0 \le a \le M}\left< \frac{a}{M}  \right>^{\pm} h_{\delta}( \frac{a}{M})
=
\frac{1}{2 } \Big \{ \sum_{\substack{n \ge 1 }} \left(\hat h_{\delta}(-n) \pm \hat h_{\delta}(n) \right ) \frac{a(n)}{n}  \Big \} \\
+\text{explicit error term depending on $q$ and $M$,}
\end{multline}
where $\hat h_{\delta}(n)$ stands for the $n$-th Fourier coefficient of the periodic function $h_{\delta}$.

As the functions $h_{\delta}$ approach the characteristic function of $[0, x]$ as $\delta \rightarrow 0$,  $\hat h_{\delta}(n)$ approaches $(1-e^{-2 \pi i nx})/(2 \pi i n)$. Applying this to \eqref{pretemp} and using the explicit form of the error term, we prove Theorem~\ref{MAIN}. 
The details of this final computation are shown in Section~\ref{PfMAIN}.

\medskip
{\bf Acknowledgements} We thank A. Cowan, D. Goldfeld, J. Louko, P. Michel, Y. Petridis, M.
Radziwill, M. Risager, F. Str\"omberg and C. Wuthrich for helpful
discussions and feedback. 
Part of the first author's work was done
during visits at the University of Patras and at Max-Planck-Institut
f\"ur Mathematik whose hospitality he acknowledges. The third author thanks RIKEN iTHEMS for their hospitality where part of the first author's work was done.
The fourth author was supported by a Royal Society University Research Fellowship.

\section{An expression of $G_M^\pm (x)$}
\label{test}

For a fixed $x \in [0, 1]$, consider the characteristic function $1_{[0, x]}$ of $[0, x]$ extended to $\mathbb{R}$ periodically with period $1$.
We will construct a family of complex valued smooth functions on $\mathbb{R}/\mathbb{Z}$ 
approximating $1_{[0, x]}$.

Let $\phi: \mathbb R \to \mathbb R$ be a smooth, non-negative function, compactly supported in $(-1/4, 1/4)$ with $\int_{-1/2}^{1/2} \phi(t)dt=1$ and $\phi(0)=1$.
For each $\delta<1$ and $t \in (-1/2, 1/2)$, set
\begin{equation}\label{phi_delta}
\phi_{\delta}(t)=\delta^{-1} \phi(t/\delta)
\end{equation}
and extend this to $\mathbb R$ periodically, with period $1$.  
The approximating functions are $h_{\delta}$ defined by 
$$h_{\delta}(t):=1_{[-\delta, x+\delta]}\star \phi_{\delta}(t)
=\int_{-\delta}^{x+\delta}\phi_{\delta}(t-v)\; dv
=\int_{t-x-\delta}^{t+\delta}\phi_{\delta}(v)\; dv, $$
where $\star$ denotes the convolution. 
This function is smooth,  periodic and satisfies $0 \le h_{\delta}(t) \le 1$. 
Further,
\begin{equation}\label{vanishing}
h_{\delta}(t)=0 \qquad \text{for $(5\delta/4+x, 1-5\delta/4)$ and its translates.}
\end{equation} 
Indeed, for $5\delta/4+ x < t < -5\delta/4 +1$, 
we have 
$\delta/4<t-x-\delta<t+\delta<1-\delta/4$. 
Since the support of $\phi_{\delta}(v)$ is contained in $(-\delta/4, \delta/4)$ 
and its translations, 
\eqref{phi_delta} implies that $\phi_{\delta}(t)$ vanishes in that range. 

We further have
\begin{equation}\label{char}
h_{\delta}(t) = 1, \quad\quad \text{ for $t\in [0, x]$}
\end{equation}
and 
\begin{equation}\label{Fourier}
\widehat{h_{\delta}}(n)
= \int_{-1/2}^{1/2} h_{\delta}(x) e^{-2\pi in x} \; dx
= \widehat{1_{[-\delta, x+\delta]}}(n) \cdot \widehat{\phi_{\delta}}(n), 
\end{equation}
for the corresponding $n$th Fourier coefficients. 
This implies that, for $n \ne 0$, 
\begin{multline}\label{Fourfor}
\widehat{h_{\delta}}(n)
=\frac{e^{2 \pi i n \delta}-e^{-2 \pi i n (x+\delta)}}{2 \pi i n} 
\int_{-1/2}^{1/2}\phi_{\delta}(t)e^{-2 \pi i n t}\; dt
=\frac{e^{2 \pi i n \delta}-e^{-2 \pi i n (x+\delta)}}{2 \pi i n} 
\int_{\frac{-1}{2 \delta}}^{\frac{1}{2\delta}}\phi(t) e^{-2 \pi i n \delta t}\; dt 
\\ =
\frac{e^{2 \pi i n \delta}-e^{-2 \pi i n (x+\delta)}}{2 \pi i n} 
\int_{-\frac12}^{\frac12}\phi(t)e^{-2 \pi i n \delta t}\; dt. 
\end{multline}
The last equality follows because  $\phi$ is supported in $(-1/4, 1/4)$. 
With the smoothness of $h_{\delta}$ we deduce that, for each $K \ge 0$ and $n \ne 0$,
\begin{equation}\label{unifbou}
|\widehat{h_{\delta}}(n)| \ll_K (|n|+1)^{-1}(\delta(1+ |n|))^{-K}.
\end{equation}
This inequality combines a bound that is uniform in $\delta$ with a stronger one that, however, is not uniform in $\delta$.  
Let
\begin{equation}\label{e:Ahpm}
\frac{1}{2} A_h^\pm (M)
=\frac{1}{M}\sum_{0 \le a \le M}  \left<\frac{a}{M}\right>^{\pm} h\left(\frac{a}{M}\right). 
\end{equation}
With these notations, we have the following
\begin{lem}\label{weights}
For $M>1$, consider any fixed $\delta=\delta_{M} <1$. Then,  
\[G^{\pm}_M(x)
= \frac{1}{2} A^\pm_{h_\delta}(M)
+ \mathcal{O}\bigg(\delta_M M^{\frac{1}{2}} q^{\frac{1}{4}} (qM)^{\epsilon} 
\prod_{\substack{p\mid M \\ \ord_p(M) <\ord_p(q)}} p^{\frac{1}{4}} 
\bigg). \]
Note that the product over primes $p\mid M$ equals $1$ if $q$ is square-free or $\gcd(q, M)=1$. 
\end{lem}
\begin{proof} 
If $a \le Mx$, then $\frac{a}{M} \le x$ and thus $h_{\delta}(t) = 1$ by \eqref{char}. These terms give us $G_M^{\pm}(x)$. 

The error term is obtained by studying the case $xM < a \leq xM +\frac{5}{4} M \delta_M$. Then
$x< \frac{a}{M} \leq x+\frac{5}{4}\delta_M$. 
By definition, $\left< \frac{a}{M} \right>^{\pm}$ is a linear combination of 
$\int_{\infty}^{a/M}f(z)\; dz$ and its complex conjugate (see \eqref{Lms}). 
In \eqref{modsym}, we prove a bound for this modular symbol that implies 
\[\left<\frac{a}{M}\right>^\pm h_{\delta}\left(\frac{a}{M}\right)
\ll M^{\frac{1}{2}} q^{\frac{1}{4}} (Mq)^{\epsilon} \prod_{\substack{p\mid M \\ \ord_p(M) <\ord_p(q)}} p^{\frac{1}{4}}, \]
and thus
\begin{multline}
\frac{1}{M}\sum_{Mx < a \leq Mx+\frac54 M \delta_M} 
\left<\frac{a}{M}\right>^{\pm} h_{\delta}\left(\frac{a}{M}\right) 
\ll \frac{1}{M} M^{\frac{1}{2}} q^{\frac{1}{4}} (qM)^{\epsilon} 
M\delta_M 
 \prod_{\substack{p\mid M \\ \ord_p(M) <\ord_p(q)}} p^{\frac{1}{4}} \\ 
= M^{\frac{1}{2}} q^{\frac{1}{4}} (qM)^{\epsilon} \delta_M
 \prod_{\substack{p\mid M \\ \ord_p(M) <\ord_p(q)}} p^{\frac{1}{4}}.  
\end{multline}
Similarly, 
\[\frac{1}{M}\sum_{M-\frac{5}{4}M\delta_M < a \leq M}
\left<\frac{a}{M}\right>^{\pm} h_{\delta}\left(\frac{a}{M}\right) 
\ll M^{\frac{1}{2}} q^{\frac{1}{4}} (qM)^{\epsilon} \delta_M
 \prod_{\substack{p\mid M \\ \ord_p(M) < \ord_p(q)}} p^{\frac{1}{4}}. \]

If $xM+ \frac{5}{4} M \delta_M < a \leq M-\frac{5}{4} M \delta_M$, 
then $x+\frac{5}{4}\delta_M < \frac{a}{M} \leq 1-\frac{5}{4}\delta_M$ 
and thus,  by \eqref{vanishing},
$h_{\delta}(a/M)$ vanishes.  
Therefore 
\begin{align*}
\frac{1}{M}& \sum_{0 \le a \le M} \left< \frac{a}{M}  \right>^{\pm} h_{\delta}\left(\frac{a}{M}\right) \\
& = 
\frac{1}{M} \left ( \sum_{0 \le a \le Mx}+\sum_{xM < a \le Mx+\frac54 M \delta_M}+ \sum_{Mx+\frac54 M \delta_M < a \le M-\frac54 M \delta_M}
+ \sum_{M-\frac54 M \delta_M < a \le M}  \right )   \left< \frac{a}{M} \right>^{\pm} h_{\delta}\left(\frac{a}{M}\right) \\
& =
\frac{1}{M}\sum_{0 \le a \le Mx} \left< \frac{a}{M}  \right>^{\pm} \cdot 1
+ \mathcal{O}\bigg(\delta_M M^{\frac{1}{2}} q^{\frac{1}{4}} (qM)^{\epsilon} 
  \prod_{\substack{p\mid M \\ \ord_p(M) < \ord_p(q)}} p^{\frac{1}{4}} \bigg)
\end{align*}
as required.
\end{proof}
In view of this lemma, we will initially study this average for an arbitrary smooth periodic $h$.
For each smooth $h: \mathbb R /\mathbb Z \to \mathbb{C}$ 
and each positive integer $M$,  we have
\begin{equation}\label{newI} 
\frac{1}{2} A_h^\pm (M)
=\sum_{n \in \ZZ} \hat h(n)  \frac{1}{M}\sum_{0 \le a \le M}  \left< \frac{a}{M}  \right>^{\pm} e^{2\pi in \frac{a}{d}}
= \sum_{n\in \ZZ} \hat{h}(n) \frac{1}{M} \sum_{d\mid M} \sum_{\substack{a\bmod{d} \\ \gcd(a, d)=1}} 
\left<\frac{a}{d}\right>^\pm e^{2\pi in \frac{a}{d}}.
\end{equation}

 We will express the right-hand side of \eqref{newI} in terms of additive twists of the $L$-function of $f$, whose definition we now recall. Let $f$ be a cusp form of weight $k$ for $\Gamma_0(q)$. 
For a positive integer $d$ and $a\in \ZZ$, let 
\[L(s, f, a/d) = \sum_{n=1}^\infty \frac{a(n) e^{2\pi in \frac{a}{d}}}{n^s}\]
be the additive twist of the $L$-function for $f$, and 

\begin{equation}
\label{e:Lambda}
\Lambda(s, f, a/d) = \int_0^\infty f\left(\frac{a}{d}+iy\right) y^s \; \frac{dy}{y}
= \left(2\pi\right)^{-s} \Gamma\left(s\right) L(s, f, a/d).
\end{equation}
The series defining $L(s, f, a/d)$ is sometimes called a Voronoi series. It converges absolutely for $\Re(s)>1+(k-1)/2$.  For consistency with the formulation of the Mazur-Rubin-Tate conjecture, we normalise the series so that the central point is at $k/2.$ 

Both $L(s, f, a/d)$ and $\Lambda(s, f, a/c)$ have analytic continuation to $s\in \C$.
Further properties are studied in Section~\ref{addtwi}. 

With the notations above, we have 
\begin{align}\label{Lms}
\left<\frac{a}{d}\right>^{\pm}
& = -\pi \int_{\infty}^0 \left ( f( \frac{a}{d}+iy) \pm f(- \frac{a}{d}+iy)\right ) \; dy\\ 
& =  \pi   \left( \Lambda(1, f, \frac{a}{d}) \pm \Lambda(1,f, -\frac{a}{d})\right )
= \frac{1}{2}\bigg(L(1, f,  \frac{a}{d}) \pm L(1, f,  -\frac{a}{d})\bigg). \nonumber
\end{align}
Here we used $\overline{f( \frac{a}{d}+ix)}=f(-\frac{a}{d}+ix)$.
This implies
\begin{equation}\label{unfol}
\sum_{\substack{a \mod d \\ \gcd(a, d)=1}}
\left< \frac{a}{d} \right>^{\pm} e^{2 \pi i n \frac{a}{d}}
=\pi \sum_{\substack{a \mod d \\ \gcd(a, d)=1}}
\left ( \Lambda(1, f, \frac{a}{d}) \pm \Lambda(1,f, -\frac{a}{d})\right ) e^{2 \pi i n \frac{a}{d}}. 
\end{equation}

Applying \eqref{unfol} to \eqref{newI}, 
\begin{equation}
A_h^\pm (M)
= \sum_{n\in \ZZ} \hat{h}(n) \frac{1}{M} \sum_{d\mid M} \sum_{\substack{a\bmod{d} \\ \gcd(a, d)=1}} 
\left ( L(1, f, \frac{a}{d}) \pm L(1,f, -\frac{a}{d})\right ) e^{2 \pi i n \frac{a}{d}}.
\end{equation}
Let 
\begin{equation}\label{alpha2twist}
\alpha_{n, M}(t)=\frac{1}{M} \sum_{a\bmod{M}} e^{-2\pi in \frac{a}{M}} L(t, f, \frac{a}{M})
= \frac{1}{M} \sum_{d\mid M} \sum_{\substack{a\bmod{d} \\ \gcd(a, d)=1}} e^{-2\pi in \frac{a}{d}} L(t, f, \frac{a}{d}).
\end{equation}
Then we get
\begin{equation}\label{e:ApmhM}
A_h^\pm(M) = \sum_{n\in \ZZ} \hat{h}(n) \big(\alpha_{-n, M}(1) \pm \alpha_{n, M}(1)\big). 
\end{equation}

We will study the properties of $L(t, f, a/d)$, the additive twist of an $L$-function twists in the next section. 
As mentioned in the introduction, we prove our results for general levels and weights. 
We summarize the results for the special case of interest of weight $2$ in Section~\ref{ss:add_special}.

\section{Properties of the additive twist of an $L$-function}\label{addtwi}
In this section we bound Fourier coefficients of locally contragredient newforms twisted by Dirichlet characters.
Our bounds are uniform in terms of the level and they will be crucial for the proof of the main theorem. Those twisted newforms arise in the context of a general functional equation for the additive twist of an $L$-function. Our functional equation is of independent interest because all references we are aware of give the functional equation only for special combinations of the level and the denominator of the additive twist \cite{KMV}. Recent papers by Assing and Corbett (\cite{C}, \cite{AC}) contain the proof of a very general Voronoi  summation formula which is closely related to a functional equation for $L$-series with additive twists, but not in the explicit form we need here. 

\subsection{Notations} \label{notations}
We closely follow \cite{AL}. 
Let $k$ be an integer. 
For any function $h: \bH\to \C$ and any matrix $\gamma=\sm a & b\\ c & d\esm \in \GL_2^+(\R)$, 
define 
\[(h\mid \gamma)(z) 
= \det(\gamma)^{\frac{k}{2}} (cz+d)^{-k} h\left(\frac{az+b}{cz+d}\right).\]
For a positive integer $q$ and a Dirichlet character $\xi\mod{q}$, 
let $M_k(q, \xi)$ (resp. $S_k(q, \xi)$) be the space of holomorphic modular forms (resp. cusp forms) of level $q$, weight $k$ and central character $\xi$. 
Then $f\in S_k(q, \xi)$ has the following Fourier expansion 
\[f(z) = \sum_{n=1}^\infty a(n) e^{2\pi in z}. \]
The Hecke operators $T_n$ for $\gcd(n, q)=1$, $U_d$ and $B_d$ for $d\mid q$ are given by:
\begin{align*}
& f\mid T_n = n^{\frac{k}{2}-1} \sum_{ac=n} \sum_{b=0}^{c-1} \xi(a) f\mid\bpm a & b\\ 0 & c\ebpm \\
& f\mid U_d = d^{\frac{k}{2}-1} \sum_{b=0}^{d-1} f\mid \bpm1 & b\\ 0 & d\ebpm  \\
& f\mid B_d = d^{-\frac{k}{2}} f\mid \bpm d & 0 \\ 0 & 1\ebpm. 
\end{align*}
For a primitive Dirichlet character $\chi\bmod{r}$, we define
\[f\mid R_\chi = \sum_{u\bmod{r}}\overline{\chi(u)} f\mid\bpm r & u\\ 0& r\ebpm.\]

Let $N_k(q, \xi)$ denote the set of Hecke-normalized (i.e. the first Fourier coefficient is $1$) cuspidal newforms of weight $k$ and level $q$ and central character $\xi$. 
If $f\in N_k(q, \xi)$ then $f\in S_k(q, \xi)$ is an eigenform of all Hecke operators $T_n$ for $\gcd(n, q)=1$ and $U_d$ for $d\mid q$ (\cite[p. 222]{AL}). 

For a primitive Dirichlet character $\chi\mod{r}$, we define the multiplicative twist of $f\in N_k(q, \xi)$,
\begin{equation}\label{e:f^chi}
f^\chi(z) := \sum_{n=1}^\infty a(n)\chi(n)e^{2\pi inz}
= \frac{1}{\tau(\overline{\chi})} (f\mid R_\chi)(z).
\end{equation}
where $\tau(\bar{\chi}) = \sum_{\alpha\mod{r}}\overline{\chi}(\alpha)e^{2\pi i \frac{\alpha}{r}}$ is the Gauss sum for $\bar \chi$. 
From \cite[Proposition~3.1]{AL}, we can deduce that $f^\chi\in S_k([q, \cond(\xi)r, r^2], \xi \chi^2)$. (Here $[a, b]$ stands for the leact common multiple of $a, b$.) 
We will further be using \cite[Lemma 1.4]{BLS}, where tight bounds for the level of a twist of a
newform are shown.

 It should be stressed that the twist $f^{\chi}$ need not be a newform even if $f$ is a newform and $\chi$ is primitive. The main aim of this section is to address this problem in the case of interest, by decomposing the relevant twist (acted upon by an involution) in terms of newforms. 

\subsection{The  Atkin-Lehner-Li-operator  and additive twists}\label{ss:WQ}
Assume that $R\mid q$ and $\gcd(R, q/R)=1$. 
Then a Dirichlet character $\xi$ modulo $q$ can be written as a product of Dirichlet characters $\xi_R$ modulo $R$ and $\xi_{q/R}$ modulo $q/R$, i.e., $\xi = \xi_R \xi_{q/R}$. 

Put
\begin{equation}\label{e:WQ}
W_R = \bpm R x_1 & x_2 \\ q x_3 & Rx_4\ebpm,
\end{equation}
where $x_1, x_2, x_3, x_4\in \ZZ$, 
$x_1\equiv 1\mod{q/R}$, $x_2\equiv 1\mod{R}$
and $\det(W_R)= R(Rx_1x_4-\frac{q}{R} x_2 x_3)=R$.
By \cite[Proposition~1.1]{AL}, for $f\in M_k(q, \xi)$ (resp. $S_k(q, \xi)$), 
we have $f\mid W_R\in M_k(q, \overline{\xi_R}\xi_{q/R})$ (resp. $S_k(q, \overline{\xi_R}\xi_{q/R})$) and 
\[f\mid W_R\mid W_R = \xi_R(-1) \overline{\xi_{q/R}(R)} f.\]
For $f\in S_k(q, \xi)$, let 
\begin{equation}\label{e:tildef_R}
\tilde{f}_R = f\mid W_R\in S_k(q, \overline{\xi_R} \xi_{q/R}). 
\end{equation}

The aim of this section is to prove the following theorem.
\begin{thm}\label{thm:additive_fe}
For $q, M_1\in \NN$ let 
\begin{align*}
& M=\prod_{p\mid M_1, \ord_p(M_1) \geq \ord_p(q)} p^{\ord_p(M_1)} \\
& r=\prod_{p\mid M_1, \ord_p(M_1) < \ord_p(q)} p^{\ord_p(M_1)} \\
& R=\prod_{p\mid \gcd(q, r)} p^{\ord_p(q)} \prod_{p\mid q, p\nmid M_1} p^{\ord_p(q)}.
\end{align*}
 For each $n|r$, we set
$$r_n= \prod_{p\mid r, p\nmid n} p.$$

%

For any $\alpha\mod{M_1}$, set $\alpha \equiv ar+uM\mod{M_1}$ 
for $a\mod{M}$ and $u\mod{r}$.
For a Hecke-normalized newform $f\in N_k(q, \xi)$, 
we have

\begin{multline}\label{e:additive_fe}
\Lambda \big(f, s, \frac{\alpha}{Mr}\big)=
\frac{i^k}{\varphi(r)} \sum_{\substack{n\mid r}} 
\frac{r}{n r_n}
\sum_{e\mid r_n} 
\sum_{\substack{\chi\mod{n} \\ \text{ primitive}}}
\chi(u \bar{e}) \tau(\overline{\chi})
\mu\left(\frac{r_n}{e}\right) \varphi\left(\frac{r_n}{e}\right)
\\ \times
(\xi_{R'}\chi^2)(-M)(M^2R')^{\frac{k}{2}-s} 
\frac{\overline{\xi_{q/R}}\left(\frac{r}{ne} a\right) a\left(\frac{r}{ne}\right) } 
{\left(\frac{r}{ne}\right)^{s}}
\Lambda\big(\widetilde{f^\chi}_{R'}, k-s,  -\frac{\overline{R'a\frac{r}{ne}}}{M}\big). 
\end{multline}
Here $R'=[R, \cond(\xi_R)r, r^2]$,
$\overline{R'a\frac{r}{ne}}$ is the inverse of $R'a\frac{r}{ne}$ modulo $M$ and
$\widetilde{f^\chi}_{R'} = f^\chi\mid W_{R'}$ .

\end{thm}


\subsubsection{Proof of Theorem~\ref{thm:additive_fe}}
 We first note the following elementary facts we will be using in the sequel. We have $M_1 = rM$ and $r\mid q$, with $(r, M)=1$.   Also $R\mid q$, $r\mid R$ and $(R, q/R)=1$. 
Moreover, $\frac{q}{R} \mid M$ and $r<R$, except for when $r=R = 1$ in which case $q \mid M$.

We next have the following

\begin{lem}
For $q\in \NN$, assume that $R\mid q$ and $\gcd(R, q/R)=1$. 
Take $M\in \NN$ such that $\frac{q}{R}\mid M$ and $\gcd(R, M)=1$.
For $a\mod{M}$ with $\gcd(a, M)=1$, set 
\begin{equation}\label{e:VNQMa}
V_{q, R}^{M, a} 
= \bpm R\overline{Ra} & \frac{1-Ra\overline{Ra}}{M} \\ -q \frac{M}{q/R} & Ra\ebpm
\end{equation}
be an integral matrix with $\det(V_{q, R}^{M, a})=R$. 
Here $Ra\overline{Ra}\equiv 1\mod{M}$.

When $f\in S_k(q, \xi)$, we have
\begin{equation}\label{e:f_flip_V}
f\big(\frac{a}{M}+iy\big) 
= 
\xi_R(-M) \overline{\xi_{q/R}(a)} i^{k} (MR^{\frac{1}{2}} y)^{-k}
\tilde{f}_R\big(-\frac{\overline{Ra}}{M} + i\frac{1}{M^2 Ry}\big). 
\end{equation}
\end{lem}
\begin{proof}
Applying \cite[Proposition~1.1]{AL}, 
\[\tilde{f}_R\mid V_{q, R}^{M, a}
= \overline{\xi_R}(M) \xi_{q/R}(Ra) \xi_{R}(-1) \overline{\xi_{q/R}}(R) f
= \overline{\xi_R}(-M) \xi_{q/R}(a) f. \]
Note that 
\[V_{q, R}^{M, a} \left(\frac{a}{M}+iy\right)
= -\frac{\overline{Ra}}{M} + i \frac{1}{M^2 Ry}.\]
So we get
\begin{align*}
f\big(\frac{a}{M}+iy\big) 
& = \xi_R(-M) \overline{\xi_{q/R}}(a) \big(\tilde{f}_R\mid V_{q, R}^{M, a}\big)\left(\frac{a}{M}+iy\right)\\
& = 
\xi_R(-M) \overline{\xi_{q/R}}(a) R^{\frac{k}{2}} (-iMRy)^{-k} \tilde{f}_R\big(-\frac{\overline{Ra}}{M} + i\frac{1}{M^2 Ry}\big). 
\end{align*}
\end{proof}

For $r\in \NN$ and a Dirichlet character $\chi\mod{r}$, define the generalized Gauss sum
\[c_\chi(n)= \sum_{u\mod{r}} \chi(u) e^{2\pi in \frac{u}{r}}. \]
Then by orthogonality, for $a\in \ZZ$ with $\gcd(a, r)=1$, we have
\[e^{2\pi in \frac{a}{r}} 
= \frac{1}{\varphi(r)}\sum_{\chi\mod{r}} \overline{\chi}(a)c_{\chi}(n).\]

\begin{lem}\label{e:f_Mq_chi_3}
Assume that $q$, $M_1$, $M, r$,  $r_n$  and $R$ are as in the statement of Theorem~\ref{thm:additive_fe}. 
 For any $\alpha\in \ZZ$ with $\gcd(\alpha, M_1)=1$, 
let $a\mod{M}$ and $u\mod{r}$ be suh that $\gcd(a, M)=1$, $\gcd(u, r)=1$ 
such that $\alpha\equiv aq + uM \mod{Mr}$. Then, 
%
\begin{multline}\label{e:f_Mq_chi_3}
f\big(\frac{\alpha}{Mr}+iy\big)
= 
\frac{1}{\varphi(r)}
\sum_{\substack{n\mid r}} 
\frac{r}{nr_n} 
\sum_{e\mid r_n}
a\left(\frac{r}{ne}\right)
\mu\left(\frac{r_n}{e}\right) \varphi\left(\frac{r_n}{e}\right) 
\sum_{\substack{\chi\mod{n} \\ \text{ primitive}}}
\tau(\bar{\chi}) \chi(u\bar{e}) 
f^{\chi} \big(\frac{a\frac{r}{ne}}{M}+i\frac{r}{ne}y\big). 
\end{multline}

\end{lem}
\begin{proof}
Since $\frac{\alpha}{M_1} = \frac{\alpha}{Mr} = \frac{a}{M}+\frac{u}{r}  \bmod 1 $, we get
\begin{equation}\label{e:f_Mq_chi}
f\big(\frac{\alpha}{Mr}+iy\big)
=
\sum_{m=1}^\infty a(m) e^{2\pi im \frac{u}{r}} e^{2\pi im \frac{a}{M}+iy}
= 
\frac{1}{\varphi(r)} \sum_{\chi\mod{r}} \overline{\chi}(u)
\sum_{m=1}^\infty a(m) c_{\chi}(m) e^{2\pi im\left(\frac{a}{M}+iy\right)}.
\end{equation}

For a Dirichlet character $\chi\mod{r}$, assume that $\chi$ is induced from a primitive character $\chi_*\mod{n}$. 
Let 
$r_2=\frac{r}{nr_n}$.
By \cite[Lemma~4.11]{BWBB}, we have $c_{\chi}(m)=0$ if $r_2\nmid m$ and 
for any $m\in \NN$, 
\begin{equation}\label{e:c_chi_nonzero}
c_{\chi}(mr_2) 
= r_2\chi_*(r_n) \tau(\chi_*) \overline{\chi_*}(m) \mu((r_n, m)) \varphi((r_n, m)). 
\end{equation}
Applying this to \eqref{e:f_Mq_chi}, we have
\begin{multline}\label{e:f_Mq_chi_2_0}
f\big(\frac{\alpha}{Mr}+iy\big)
= 
\frac{1}{\varphi(r)} \sum_{\chi\mod{r}} \overline{\chi}(u)
\sum_{m=1}^\infty a(mr_2) c_{\chi}(mr_2) e^{2\pi i mr_2 \left(\frac{a}{M}+iy\right)}
\\ = 
\frac{1}{\varphi(r)} \sum_{\chi\mod{r}} \overline{\chi}(u)
r_2\chi_*(r_n) \tau(\chi_*)a(r_2)
\sum_{m=1}^\infty a(m) \overline{\chi_*}(m) \mu((r_n, m)) \varphi((r_n, m))
e^{2\pi i m\left(r_2\left(\frac{a}{M}+iy\right)\right)}.
\end{multline}
The last equality holds because $f\in N_k(q, \xi)$, so $f\mid U_p=a(p)f$ for any prime $p\mid q$, so $a(mr_2) = a(r_2)a(m)$. 
Note that $r_2\mid r$ and $r\mid q$ so $r_2\mid q$.
By definition, $r_n$ is square-free, so we have
\begin{multline}\label{e:f_Mq_chi_2_1}
\sum_{m=1}^\infty a(m) \overline{\chi_*}(m) \mu((r_n, m)) \varphi((r_n, m))
e^{2\pi i m r_2 z}
= \sum_{e\mid r_n} \mu(e) \varphi(e)
\sum_{m=1}^\infty a(em) \overline{\chi_*}(em) e^{2\pi iem r_2 z}
\\ =
\sum_{e\mid r_n} \mu(e) \varphi(e) a(e) \overline{\chi_*}(e)
\sum_{m=1}^\infty a(m) \overline{\chi_*}(m) e^{2\pi i emr_2z}
= \sum_{e\mid r_0} \mu(e) \varphi(e) a(e) \overline{\chi_*}(e) f^{\overline{\chi_*}}(er_2z). 
\end{multline}
By applying \eqref{e:f_Mq_chi_2_1} to \eqref{e:f_Mq_chi_2_0} and 
taking $z=\frac{a}{M}+iy$, we get
\begin{multline*}
f\big(\frac{\alpha}{Mr}+iy\big)
= 
\frac{1}{\varphi(r)} \sum_{\chi\mod{r}} \overline{\chi(u)} 
r_2 \tau(\chi_*)a(r_2)
\sum_{e\mid r_n} \mu(e) \varphi(e) a(e) \chi_*\left(\frac{r_n}{e}\right)
f^{\overline{\chi_*}}\left(er_2\left(\frac{a}{M}+iy\right)\right)
\\ = 
\frac{1}{\varphi(r)}
\sum_{\substack{n \mid r}} 
\frac{r}{nr_n} \sum_{e\mid r_0} a\left(\frac{r}{ne}\right)
\mu\left(\frac{r_n}{e}\right) \varphi\left(\frac{r_n}{e}\right) 
\sum_{\substack{\chi\mod{n} \\ \text{ primitive}}}
\tau(\bar{\chi}) 
\chi(u\bar{e}) 
f^{\chi} \big(\frac{a\frac{r}{ne}}{M}+i\frac{r}{ne}y\big). 
\end{multline*}
\end{proof}

Now we are ready to prove Theorem~\ref{thm:additive_fe}. 
Let $n\mid r$ and let $\chi$ be a primitive Dirichlet character $\mod n.$
By \cite[Proposition~3.1]{AL}, 
$f^\chi\in S_k(R' q/R, \xi\chi^2)$ and thus 
by \eqref{e:f_flip_V}, 
\begin{multline}\label{e:f^chi_flip}
f^\chi\big(\frac{\frac{r}{ne} a}{M} + i \frac{r}{ne}y\big)
\\ =
i^k (\xi_{R'}\chi^2)(-M) \overline{\xi_{q/R}}\left(\frac{r}{ne} a\right) 
\big(M{R'}^{\frac{1}{2}} \frac{r}{ne} y\big)^{-k}
\widetilde{f^{\chi}}_{R'}\big(-\frac{\overline{R'a\frac{r}{ne}}}{M} + i \frac{1}{M^2R'\frac{r}{ne}y}\big).
\end{multline}
 Recall that 
$\widetilde{f^\chi}_{R'} = f^{\chi}\mid W_{R'}\in S_k(R'q/R, \overline{\xi_{R'}\chi^2}\xi_{q/R})$.
%
%

%
%
%
Applying Lemma \ref{e:f_Mq_chi_3}, we get
\begin{multline*}
\Lambda (s, f, \frac{\alpha}{Mr})
= 
\int_0^\infty f\left(\frac{\alpha}{Mr}+iy\right) y^s \; \frac{dy}{y}
\\ =
\frac{1}{\varphi(r)} \sum_{\substack{n \mid r}} 
\frac{r}{nr_n} \sum_{e\mid r_n} 
\sum_{\substack{\chi\mod{n} \\ \text{ primitive}}}
\chi(u \bar{e}) \tau(\overline{\chi})
\mu\left(\frac{r_n}{e}\right) \varphi\left(\frac{r_n}{e}\right)
a\left(\frac{r}{ne}\right) 
\int_0^\infty f^\chi\big(\frac{\frac{r}{ne} a}{M} + i \frac{r}{ne}y\big)
y^s \; \frac{dy}{y}.
\end{multline*}
By \eqref{e:f^chi_flip}, 
\begin{multline*}
\int_0^\infty f^\chi\big(\frac{\frac{r}{ne} a}{M} + i \frac{r}{ne}y\big)
y^s \; \frac{dy}{y}
\\ =
i^k (\xi_{R'}\chi^2)(-M) \overline{\xi_{q/R}}\left(\frac{r}{ne} a\right) 
\int_0^\infty \big(M{R'}^{\frac{1}{2}}\frac{r}{ne} y\big)^{-k}
\widetilde{f^{\chi}}_{R'}\big(-\frac{\overline{R'a\frac{r}{ne}}}{M} + i \frac{1}{M^2R'\frac{r}{ne}y}\big)
y^{s} \; \frac{dy}{y}\\ = 
i^k (\xi_{R'}\chi^2)(-M)(M^2R')^{\frac{k}{2}-s} 
\frac{\overline{\xi_{q/R}}\left(\frac{r}{ne} a\right)} 
{\left(\frac{r}{ne}\right)}^s
\Lambda(k-s, \widetilde{f^\chi}_{R'}, -\frac{\overline{R'a\frac{r}{ne}}}{M}). 
\end{multline*}
This implies \eqref{e:additive_fe}.
%

\subsection{Decomposition of $\widetilde{f^\chi}_{R'}$ and its Fourier coefficients}
 In this section, we restrict to the case of trivial central character $\xi$ which is the case we need for the proof of our main theorem. We do so to avoid further complicating the presentation. The results, appropriately adjusted, hold for general central characters too.

The aims of this section are to decompose  $\widetilde{f^\chi}_{R'}$ in terms of newforms and to bound its Fourier coefficients. The former aim will be achieved by Lemma~\ref{lem:decompose_f^chi} and Proposition~\ref{prop:decompose_f^chi1}, whereas the latter is the subject of Proposition~\ref{prop:Fourier_upper}.

We first fix some notation we will be using throughout the section: 
\begin{itemize}
\item $q\in \NN$;
\item $r\mid q$ and for any prime $p\mid r$, $\ord_p(r) < \ord_p(q).$ (Thus, if $q$ is square-free then $r=1$);
\item $R\mid q$ such that $r\mid R$ and $(R, q/R)=1$;
\item $\chi$ is a primitive Dirichlet character modulo $r_*\mid r$. (When $r_*=1$ then $\chi=1$); 
\item $R'=[R, r^2]$;
\item $R_*$ is the $r_*$-primary factor of $q$, i.e., $R_*=\prod_{p\mid r_*} p^{\ord_p(q)}$. 
\end{itemize}
With these notations we have
\begin{lem}\label{lem:decompose_f^chi}
Let $f$ be a Hecke-normalized newform $f\in N_k(q)$. 
Then there exist $q'\mid [q, r_*^2]$ with $\frac{q}{R_*}\mid q'$ and $F_{\chi}\in N_k\left(q', \chi^2\right)$, 
such that 
\[f^\chi(z) = \sum_{\ell\mid r_*} \mu(\ell) (F_{\chi}\mid U_\ell\mid B_{\ell})(z).\]
We set 
\begin{equation}
F_{\chi}(z) = \sum_{n=1}^\infty a_{\chi}(n) e^{2\pi inz}
\end{equation}
and a multiplicative function $\beta_{F_\chi}$: for each prime $p\mid r_*$ satisfying $F_\chi\mid U_p \neq 0$
\begin{equation}\label{e:beta_Fchi}
\beta_{F_\chi}(p^j) = \begin{cases}
1 & \text{ if } j=0 \\
-a_\chi(p) & \text{ if } j=1 \\
-p^{k-1}\chi^2(p)& \text{ if } j=2 \text{ and } p\nmid q',  \\
0 & \text{ otherwise.}
\end{cases}
\end{equation}
Let 
\[r_{*0} = \prod_{\substack{p\mid r_*, p\nmid q' \\ F_\chi\mid U_p\neq 0}} p^2 
\prod_{\substack{p\mid \gcd(r_*, q') \\ F_\chi\mid U_p\neq 0}} p.\]
Then 
\begin{equation}\label{lincomb}
f^\chi = \sum_{\ell\mid {r_{*0}}} \beta_{F_\chi}(\ell) F_\chi\mid B_\ell
\in S_k([q, r^2], \chi^2).
\end{equation}
\end{lem}
\begin{proof}
The proof of the first assertion is based on a repeated use of \cite[Theorem~3.2]{AL}. For each $p\mid r_*$, let $\chi_p$ be the primitive Dirichlet character of conductor $p^{\ord_p(r_*)}$ so that $\chi = \prod_{p\mid r_*} \chi_p$. 
By \cite[Theorem~3.2]{AL}, there exists a newform $F_{\chi_p}\in N_k(q_p', \chi_p^2)$, for some level $q_p'$  such that  $(q/p^{\ord_p(q)}) \mid q_p'$ 
\[f^{\chi_p} = F_{\chi_p}-F_{\chi_p}\mid U_p \mid B_p.\]
Further, by \cite[Lemma~1.4]{BLS}, we know that $q_p'\mid [q, p^{2\ord_p(r_*)}]$. 
If $\ell\neq p$ is a prime divisor of $r_*$, then, recalling the notations introduced in Section \ref{notations}, 
\[f^{\chi_p\chi_\ell} = F_{\chi_p}^{\chi_\ell} - \frac{1}{\tau(\overline{\chi_\ell})} F_{\chi_p} \mid U_p \mid B_p \mid R_{\chi_\ell}.\]
It is easy to see that $F_{\chi_p}| U_p | B_p | R_{\chi_\ell}
= \chi_\ell(p) F_{\chi_p} | U_p | R_{\chi_\ell} | B_p.$ Also, by \cite[Proposition~3.3]{AL}, 
\[F_{\chi_p} \mid U_p \mid R_{\chi_\ell} = \overline{\chi_\ell(p)} F_{\chi_p}\mid R_{\chi_\ell} \mid U_p
= \overline{\chi_\ell(p)} \tau(\overline{\chi_\ell}) F_{\chi_p}^{\chi_\ell} \mid U_p.\]
So we finally get 
\[f^{\chi_p\chi_\ell}
= F_{\chi_p}^{\chi_\ell}-F_{\chi_p}^{\chi_\ell} \mid U_p \mid B_p.\]
In the same way, we apply \cite[Theorem~3.2]{AL} to $F_{\chi_p}$ to deduce that there exists a $F_{\chi_p\chi_\ell}\in N_k(q_{p\ell}', (\chi_p\chi_\ell)^2)$, for some $q_{p\ell}'|[q_{p\ell}, p^{2\ord_p(r_*)}p^{2\ord_\ell(r_*)}]$ 
 with $(q/p^{\ord_p(q)} \ell^{\ord_{\ell}(q)}) \mid q_{p \ell}'$ such that 
\[F_{\chi_p}^{\chi_\ell}
= F_{\chi_p\chi_\ell}- F_{\chi_p\chi_\ell}\mid U_\ell \mid B_\ell.\]
This implies
\[f^{\chi_p\chi_\ell} = F_{\chi_p\chi_\ell}\mid (I_2-U_\ell\mid B_\ell)(I_2 - U_p\mid B_p) . \]
where $F\mid I_2 = F$. Continuing in the same way, we obtain 
\[f^\chi = F_\chi \mid \big(\prod_{p\mid r_*} \big[I_2-U_p\mid B_p\big]\big)
= \sum_{\ell\mid r_*} \mu(\ell) (F_{\chi}\mid U_\ell\mid B_{\ell}).\]
for some $F_{\chi}\in N_k\left(q', \chi^2\right)$ and some $q'\mid [q, r_*^2]$. 

To prove \eqref{lincomb} we first observe that 
\[f^\chi 
= F_\chi\mid \big(\prod_{\substack{p\mid r_*,\\ F_\chi\mid U_p \neq 0}} \big[I_2-U_p\mid B_p\big]\big).\]
Now, if $p\mid r_*$ and $p \nmid q'$, then, by the definition of $U_p$ and by $F_{\chi}|T_p=a_{\chi}(p)F_{\chi}$
we have
\begin{equation*}\label{FUp}
F_{\chi}\mid U_{p} = a_{\chi}(p) F_{\chi} - p^{k-1} \chi^2(p) F_{\chi}\mid B_p.\end{equation*}
If, on the other hand, $p\mid r_*$ and $p\mid q'$, then $\chi(p)=0$ and $F_\chi\mid U_p = a_\chi(p) F_\chi$.
Thus 
\begin{multline}
f^\chi 
= F_\chi\mid \big(\prod_{\substack{p\mid r_*, p\nmid q' \\ F_\chi\mid U_p\neq 0}} \big[I_2-a_\chi(p) B_p - p^{k-1} \chi^2(p) B_{p^2}\big]
\prod_{\substack{p\mid \gcd(r_*, q') \\ F_\chi\mid U_p \neq 0}} \big[I_2-a_\chi(p)B_p\big]\big)
\\ = \sum_{\ell\mid r_{*0}} \beta_{F_\chi}(\ell) F_\chi\mid B_{\ell}. 
\end{multline}
\end{proof}

We can use this lemma to prove 
\begin{prop}\label{prop:decompose_f^chi1} 
With the notation fixed in the beginning of the section,
let $f$ be a Hecke-normalized newform $f\in N_k(q)$ and let  $F_{\chi}$ be the newform in 
$N_k\left(q', \chi^2\right)$ (for some $q'\mid [q, r_*^2]$) as in Lemma \ref{lem:decompose_f^chi}.
Let $R_*'$ be the $(r_*, q')$-primary factor of $q'$ and set $Q_* = \frac{R_* R'}{R_*' R r_{*0}}$. 

Then $R'/R_*' r_{*0}, \, \, Q_* \in \ZZ$ and 
\begin{equation}\label{e:f^chi_F_chi_Q}
\widetilde{f^\chi}_{R'}(z) = f^\chi \mid W_{R'}(z)
= \sum_{\ell\mid r_{*0}} \beta_{F_\chi}(\ell) \ell^{-\frac{k}{2}} 
\left(Q_*\frac{r_{*0}}{\ell}\right)^{\frac{k}{2}} \widetilde{F_\chi}_{\frac{R R_*'}{R_*}} \left(Q_*\frac{r_{*0}}{\ell} z\right).
\end{equation}
where $\widetilde{F_\chi}_{\frac{R R_*'}{R_*}} = F_{\chi}\mid W_{\frac{R R_*'}{R_*}}$. Further 
there exists $\lambda_{\frac{R R_*'}{R_*}}(F_\chi) \in \mathbb C$ of absolute value one
such that 
\[\overline{\lambda_{\frac{R R_*'}{R_*}}(F_\chi)} \widetilde{F_{\chi}}_{\frac{R R_*'}{R_*}} 
\in N_k\left(R_*' q/R_*, \overline{\chi}^2\right).\]
(The constant $\lambda_{\frac{RR_*'}{R_*}}(F_{\chi})$ is an {\rm Atkin-
Lehner-Li pseudo eigenvalue}.) 
\end{prop}

\begin{proof}
We first easily see using the definitions of the invariants involved that
\begin{equation}
\label{q'}
R_*'\frac{q}{R_*}=q'
\end{equation}

We next prove that $R_*' r_{*0}\mid R'$.
Since $R_*'\mid R'$, we only need to check that $\ord_p(R_*'r_{*0}) \leq \ord_p(R')$ for each prime $p\mid r_{*0}$. 
Take a prime $p\mid r_{*0}$. 
By definition this implies that $F_\chi\mid U_p\neq 0$, which, by \cite[Corollary~3.1]{AL}, is equivalent to either
\begin{itemize}
\item $p\nmid R_*'$, or
\item $p\| R_*'$, or
\item $p^2\mid R_*'$ and $\ord_p(\cond(\chi^2)) = \ord_p(R_*')$.
\end{itemize}
Recall that $p\mid r_{*0}$ implies that $p\mid r_*$ so $p\mid r$. 
Since $R'=[R, r^2]$, we have $p^2\mid R'$

Now we consider each case with the prime $p\mid r_{*0}$.  
When $p\nmid R_*'$ then $\ord_p(r_{*0})=2$ so $\ord_p(R_*' r_{*0}) = 2\leq \ord_p(R')$. 
When $\ord_p(R_*')=1$ then $\ord_p(r_{*0})=1$, so $\ord_p(R_*' r_{*0}) = 2\leq \ord_p(R')$. 
When $\ord_p(R_*')\geq 2$ and $\ord_p(\cond(\chi^2))=\ord_p(R_*')$, we first note that 
\[\ord_p(R_*')=\ord_p(\cond(\chi^2)) \leq \ord_p(r_*)\leq \ord_p(r).\]
Moreover $\ord_p(r_{*0})=1$. 
So we get
\[\ord_p(R_*' r_{*0}) \leq \ord_p(r)+\ord_p(r_{*0}) = \ord_p(r)+1\leq 2\ord_p(r) \leq \ord_p(R'). \]
Therefore, we conclude that $R_*' r_{*0}\mid R'$.

We can use this to verify the integrality of $Q_*$. We have  $\frac{R}{R_*}\in \ZZ$ and $\frac{R}{R_*}\mid R'$.
Moreover $(R/R_*, R_*' r_{*0})=1$. 
So $Q_* = \frac{R_* R'}{R_*' R r_{*0}} = \frac{R'}{\frac{R}{R_*} R_*' r_{*0}} \in \ZZ$. 

Finally, we derive a formula for 
\[f^\chi \mid W_{R'} = \sum_{\ell\mid r_{*0}}\beta_{F_\chi}(\ell) F_\chi\mid B_\ell\mid W_{R'}.\]
Since, as shown above, $R_*' r_{*0}\mid R'$, we have $\ell \mid R'$ for each $\ell \mid r_{*0}.$ 
Then, by \cite[Proposition~1.5]{AL}, \[F_\chi\mid B_\ell \mid W_{R'} = \ell^{-\frac{k}{2}} F_\chi\mid W_{\frac{R'}{\ell}}.\]
Note that $W_{R'}$-operator on the left-hand side is an operator for level $R' \frac{q}{R}$ 
and $W_{\frac{R'}{\ell}}$-operator on the right-hand side is an operator for level $\frac{R'}{\ell} \frac{q}{R}$. 

Set 
\[W_{\frac{R'}{\ell}} = \bpm \frac{R'}{\ell} x_1 & x_2 \\ \frac{q}{R} \frac{R'}{\ell} x_3 & \frac{R'}{\ell}x_4\ebpm, \]
where $x_1, x_2, x_3, x_4\in \ZZ$, $\det(W_{R'/\ell})=R'/\ell$, $x_1\equiv 1\mod{q/R}$ and $x_2\equiv 1\mod{R'/\ell}.$ 

Since, by \eqref{q'}, $F_\chi\in N_k(R_*' q/R_*, \chi^2)$,
 we lower the level of $W_{\frac{R'}{\ell}}$ to $R_*'\frac{q}{R_*}$:
\[W_{\frac{R'}{\ell}}
= W_{R_*'\frac{R}{R_*}}\bpm Q_*\frac{r_{*0}}{\ell} & \\ & 1\ebpm 
\quad  
\text{where} \, \, 
W_{R_*' \frac{R}{R_*}} = \bpm \frac{R}{R_*} R_*' x_1 & x_2\\ R_*' \frac{q}{R_*} x_3 & \frac{R'}{\ell} x_4 \ebpm.\]
Here $W_{\frac{R R_*'}{R_*}}$ is an operator for level $R_*' q/R_*$ 
and we get
\[(F_\chi\mid W_{\frac{R'}{\ell}})(z) 
= \left(F_\chi\mid W_{\frac{R R_*'}{R_*}} \mid \bpm Q_*\frac{r_{*0}}{\ell} & \\ & 1\ebpm \right)(z)
= \left(Q_*\frac{r_{*0}}{\ell}\right)^{\frac{k}{2}} \widetilde{F_\chi}_{\frac{R R_*'}{R_*}} \left(Q_*\frac{r_{*0}}{\ell}z\right).\]
This implies \eqref{e:f^chi_F_chi_Q}. 
Finally, by \cite{AL}, there exists a constant $\lambda_{\frac{RR_*'}{R_*}}(F_\chi)$ of absolute value one, such that 
\[\overline{\lambda_{\frac{R R_*'}{R_*}}(F_\chi)} \widetilde{F_\chi}_{\frac{RR_*'}{R_*}}
\in N(R_*' q/R_*, \overline{\chi}^2, k).\]
\end{proof}

 The above lemma and proposition allow us to prove good bounds the Fourier coefficients of $\widetilde{f^\chi}_{R'}(z) $:  
\begin{prop}\label{prop:Fourier_upper}
With the notations in Proposition~\ref{prop:decompose_f^chi1}, set
\begin{equation}\label{e:f_bchiR'}
\widetilde{f^\chi}_{R'}(z) = \sum_{m=1}^\infty b_{\chi, R'}(m)e^{2\pi imz}.
\end{equation}
Then $b_{\chi, R'}(m)=0$ when $Q_* \nmid m$, 
and otherwise, for $m\in \NN$, 
\begin{equation}\label{e:bchiR'_upper}
\left|\left(Q_* m\right)^{-\frac{k-1}{2}} b_{\chi, R'}\left(Q_* m\right)\right|
  \ll_{\epsilon} 
\left(\frac{m}{r_{*0}}\right)^{\epsilon}
\left(Q_* r_{*0} \right)^{\frac{1}{2}} 
\sigma_{-1+2\epsilon}(r_{*0}),
\end{equation}
for any $\epsilon>0$.  
\end{prop}
\begin{proof}
Applying Lemma~\ref{lem:decompose_f^chi}, we set the Fourier expansion of 
$\widetilde{F_{\chi}}_{\frac{R R_*'}{R_*}}$ to be
\begin{equation}\label{e:widetildeFchiQ_Fourier}
\widetilde{F_{\chi}}_{\frac{R R_*'}{R_*}}(z)
= \lambda_{\frac{R R_*'}{R_*}}(F_\chi) \sum_{n=1}^\infty \ta_{\chi, R_*'\frac{R}{R_*}}(n) e^{2\pi inz}. 
\end{equation}
By \cite[(1.1)]{AL}, we get
\begin{equation}
\tilde{a}_{\chi, R_*'\frac{R}{R_*}}(p)
= \begin{cases} 
\overline{\chi}^2(p) a_\chi(p) & \text{ if } p\nmid R_*'\frac{R}{R_*} \\
\overline{a_{\chi}(p)} & \text{ if } p\mid R_*' \frac{R}{R_*}.
\end{cases}
\end{equation}

We then apply \eqref{e:widetildeFchiQ_Fourier} to \eqref{e:f^chi_F_chi_Q} to get
\begin{equation}
\widetilde{f^\chi}_{R'}(z) 
 = \lambda_{R_*'\frac{R}{R_*}}(F_\chi)
\sum_{\ell\mid r_{*0}} \beta_{F_\chi}(\ell) \ell^{-\frac{k}{2}} 
\left(Q_*\frac{r_{*0}}{\ell}\right)^{\frac{k}{2}} 
\sum_{n=1}^\infty \tilde{a}_{\chi, R_*'\frac{R}{R_*}}(n)e^{2\pi in Q_*\frac{r_{*0}}{\ell} z}.
\end{equation}
Comparing both sides, $b_{\chi, R'}(m)=0$ when $Q_{*} \nmid m$. 
For $m\in \NN$, 
\begin{multline}
b_{\chi, R'}\left(Q_* m\right)
= \lambda_{R_*'\frac{R}{R_*}}(F_\chi)
\sum_{\substack{\ell\mid r_{*0} \\ \frac{r_{*0}}{\ell}\mid m}} \beta_{F_\chi}(\ell) \ell^{-\frac{k}{2}} 
\left(Q_*\frac{r_{*0}}{\ell}\right)^{\frac{k}{2}}
\ta_{\chi, R_*'\frac{R}{R_*}}\left(\frac{m}{r_{*0}/\ell}\right)
\\ = \lambda_{R_*'\frac{R}{R_*}}(F_\chi)
\sum_{\ell\mid \gcd(r_{*0}, m)} \beta_{F_\chi}(r_{*0}/\ell) (r_{*0}/\ell)^{-\frac{k}{2}} 
\left(Q_* \ell\right)^{\frac{k}{2}}
\ta_{\chi, R_*'\frac{R}{R_*}}\left(\frac{m}{\ell}\right).
\end{multline}
Recalling \eqref{e:beta_Fchi}, we deduce that  $b_{\chi, R'}\left(Q_* m\right)$ equals
\begin{equation*}
 \lambda_{R_*'\frac{R}{R_*}}(F_\chi)
\sum_{\ell\mid \gcd(r_{*0}, m)} 
\bigg[\prod_{p\| r_{*0}/\ell} (-p^{-\frac{k}{2}} a_{\chi}(p)) 
\prod_{p^2\| r_{*0}/\ell} (-p^{-1} \chi^2(p))\bigg]
\left(Q_* \ell \right)^{\frac{k}{2}}
\ta_{\chi, R_*'\frac{R}{R_*}}\left(\frac{m}{\ell}\right).
\end{equation*}
For any $m\in \NN$,  since $F_{\chi}$ and $\widetilde{F_{\chi}}_{\frac{R R_*'}{R_*}}(z)$ are newforms,  we have 
\[|a_\chi(m)| \ll_{\epsilon} m^{\frac{k-1}{2}+\epsilon} \quad \text{ and } \quad 
\left|\ta_{\chi, R_{*}'\frac{R}{R_*}}(m)\right| \ll_{\epsilon} m^{\frac{k-1}{2}+\epsilon}, \]
for any $\epsilon>0$.
Thus we finally get
\begin{multline*}
\left|\left(Q_* m\right)^{-\frac{k-1}{2}} b_{\chi, R'}\left(Q_* m\right)\right|
  \ll_{\epsilon}   \left(Q_* m\right)^{-\frac{k-1}{2}} 
\sum_{\ell\mid \gcd(r_{*0}, m)} 
\bigg[\prod_{p\|\frac{ r_{*0}}{\ell}} p^{-\frac{k}{2}} p^{\frac{k-1}{2}+\epsilon} 
\prod_{p^2\| \frac{r_{*0}}{\ell}} p^{-1}\bigg]
Q_*^{\frac{k}{2}}
\ell^{\frac{k}{2}} 
\left(\frac{m}{\ell}\right)^{\frac{k-1}{2}+\epsilon} 
\\ =Q_*^{\frac{1}{2}} m^{\epsilon}
\sum_{\ell\mid \gcd(r_{*0}, m)} 
\bigg[\prod_{p\| r_{*0}/\ell} p^{-\frac{1}{2}+\epsilon} 
\prod_{p^2\| r_{*0}/\ell} p^{-1}\bigg]\ell^{\frac{1}{2}-\epsilon}
\\ \leq
m^{\epsilon} 
r_{*0}^{\frac{1}{2}-\epsilon} 
Q_*^{\frac{1}{2}} 
\sum_{\ell\mid r_{*0}} 
\bigg[\prod_{p\| \ell}  p^{-\frac{1}{2}+\epsilon} 
\prod_{p^2\| \ell} p^{-1}\bigg]\ell^{-\frac{1}{2}+\epsilon}
\leq 
\left(\frac{m}{r_{*0}}\right)^{\epsilon}
(Q_* r_{*0})^{\frac{1}{2}} 
\sum_{\ell\mid r_{*0}} \ell^{-1+2\epsilon}
\\ = \left(\frac{m}{r_{*0}}\right)^{\epsilon}(Q_* r_{*0})^{\frac{1}{2}} \sigma_{-1+2\epsilon}(r_{*0}). 
\end{multline*}
\end{proof}

\subsection{Additive twists in the special case applying to Theorem~\ref{MAIN}}\label{ss:add_special}
 We now further specialize to the case of weight $2$. This is the setting of our main theorem, where we consider Hecke-normalized newforms of weight $2$ and level $q$.  By Theorem~\ref{thm:additive_fe} and Proposition~\ref{prop:Fourier_upper} we have Corollary \ref{cFE}. 
As applications of this corollary, we then obtain an upper bound for $\left|\int_0^\infty f(a/d+iy) \; dy\right|$ \eqref{modsym}
and the approximate functional equation \eqref{approx} for $L(1, f, a/d)$. 

\begin{cor}\label{cFE} 
Let $f$ be a Hecke-normalized newform of weight $2$ for level $q$. 
Let $a, d$ be coprime integers and set 
\begin{align*}
& M_d=\prod_{p\mid d, \ord_p(d) \geq \ord_p(q)} p^{\ord_p(d)}, \\
& r_d=\prod_{p\mid d, \ord_p(d) < \ord_p(q)} p^{\ord_p(d)}, \\
& R_d=\prod_{p\mid \gcd(q, r_d)} p^{\ord_p(q)} \prod_{p\mid q, p\nmid d} p^{\ord_p(q)}, \\
&  R_d' = [R_d, r_d^2]. 
\end{align*}

Further, consider $a_1 \mod{M_d}$ and $a_2 \mod {r_d}$ such that 
$a \equiv a_1 r_d+a_2 M_d\mod{d}$.
Then we have 
\begin{multline}\label{e:additive_fe_k=2}
(M_d^2R_d')^{s-1} 
\Lambda(s, f, \frac{a}{d})
\\= \frac{-1 }{\varphi(r_d)} 
\sum_{\substack{n \mid r_d,\\ \frac{r_d}{n} \text{ square-free} \\ \gcd(n, r_d/n)=1}} 
\sum_{\substack{\chi\mod{n} \\ \text{ primitive}}}
\tau(\overline{\chi})
\chi\left(a_2 \overline{\left ( \frac{r_d}{n} \right )}\right) \chi^2 (M_d)
\Lambda (2-s, \widetilde{f^\chi}_{R_d'}, -\frac{\overline{R_d'a_1}}{M_d}). 
\end{multline}
Here $\overline{R_d' a_1}$ is the inverse of  $R_d'a_1$ modulo $M_d$. 

 We also repeat the following notations for the reader's convenience.  
For a primitive Dirichlet character $\chi$ for $\cond(\chi)=r_{d*} \mid r_d$, 
denote the invariants 
\begin{align*}
& R_{d*} = \prod_{p\mid r_{d*}} p^{\ord_p(q)}, \\
& R_{d*}'= \prod_{p\mid (r_{d*}, q')} p^{\ord_p(q')}, \\
& r_{d*0} = \prod_{\substack{p\mid r_{d*}, p\nmid q', \\ F_{\chi}\mid U_p\neq 0}} p^2 
\prod_{\substack{p\mid (r_{d*}, q'), \\ F_\chi\mid U_p\neq 0}} p, \\
& Q_{d*} = \frac{R_{d*} R_{d}'}{R_{d*}' R_d r_{d*0}}, 
\end{align*}
and $b_{\chi, R_d'}(m)$  as given in  Proposition~\ref{prop:decompose_f^chi1} and Proposition~\ref{prop:Fourier_upper}. 
Then $b_{\chi, R_d'}(m)=0$ when $Q_{d*}\nmid m$ 
and for $n\in \NN$, we get
\begin{equation}\label{e:bchiR'_bound}
\left|(Q_{d*}n)^{-\frac{1}{2}} b_{\chi, R_d'}\left(Q_{d*}n\right)\right|
\ll_{\epsilon}
\left(\frac{n}{r_{d*0}}\right)^{\epsilon}
(Q_{d*} r_{d*0})^{\frac{1}{2}} 
\sigma_{-1+2\epsilon}(r_{*0}),
\end{equation}
for any $\epsilon>0$. 
\end{cor}
\begin{proof} This is just a specialization of Theorem~\ref{thm:additive_fe} and Proposition~\ref{prop:Fourier_upper} to the case $k=2$ and trivial central character $\xi$. The functional equation \eqref{e:additive_fe} simplifies in this case to 
\eqref{e:additive_fe_k=2}.

Indeed, suppose that, for some $n\mid r_d$  and $e\mid \prod_{p|r_d, r \nmid n}p$, we have $r_d \ne ne$.
Then $a(r_d/ne)=0$. 
This is because, if $p\mid \frac{r_d}{ne} \mid r_d$, then $r^2\mid q$ (by the definition of $r_d$) 
and thus $a(pm)=0$, for all $m \in \NN$, since $f$ is a newform.
Therefore, $e=r_d/n.$ 
Since $\prod_{p\mid r_d, r \nmid n} p \mid \frac{r_d}{n}$ and $e\mid \prod_{p|r_d, r \nmid n}p$, we have
$e=\prod_{p\mid r_d, r \nmid n}p=\frac{r_d}{n}$ and hence, $(n, r_d/n)=1$ and $\frac{r_n}{n}$ is square-free.
\end{proof}

 As an application of this corollary we prove the following proposition which we need for the proof of our main theorem, but which is also of independent interest.
\begin{prop}\label{lemmodsym} Let $f$  be a Hecke-normalized newform of weight $2$ for level $q$. Then, for each $\epsilon>0$, 
\begin{equation}\label{modsym} 
\left|\int_\infty^{\frac{a}{d}} f(z) \; dz\right|
=
\left|\int_0^\infty f\big(\frac{a}{d}+iy\big) \; dy\right|
\ll d^{\frac{1}{2}}q^{\frac{1}{4}} (qd)^{\epsilon} 
 \prod_{\substack{p\mid d \\ \ord_p(d) < \ord_p(q)}} p^{\frac{1}{4}}.
\end{equation}
Note that the product over $p$ equals $1$ if $q$ is square-free. 
\end{prop}

\begin{proof}
 We first observe that 
\begin{equation}\label{e:MdRd'}
M_d^2 R_d' = \frac{d^2 R_d}{\gcd(R_d, r_d^2)} = [q, d^2].
\end{equation}
%
 The second equality holds because
\[\gcd(q, d^2) 
= \gcd(R_d q/R_d, M_d^2 r_d^2)
= \gcd(q/R_d, M_d^2) \gcd(R_d, r_d^2)
= \frac{q}{R_d} \gcd(R_d, r_d^2),\]
since $\frac{q}{R_d} | M_d$. 
Thus we have
$\gcd(R_d, r_d^2) = \gcd(q, d^2)\frac{R_d}{q}$. 
%

It follows from this that on the line $\Re(t)=1+\epsilon$,  
\begin{equation}\label{e:convex_L}
 (M_d^2 {R_d'})^{\frac{t}{2}}  
\Lambda(t+\frac{1}{2}, f, \frac{a}{d})
\ll  [q, d^2]^{1/2 + \epsilon}
\end{equation}
because of the Stirling bound for the Gamma function.

Similarly, using Corollary~\ref{cFE} 
we will deduce the following bound for $t$ with $\Re(t)=-\epsilon$: 
\begin{equation}\label{convex1}
 (M_d^2 R_d')^{\frac{t}{2}} 
\Lambda(t+\frac{1}{2}, f, \frac{a}{d})
\ll
(dq)^{\epsilon} d q^{\frac{1}{2}} 
  \prod_{\substack{p\mid d \\ \ord_p(d)<\ord_p(q)}} p^{\frac{1}{2}}.
\end{equation}

This analysis is more involved, and we present most of the details.
For $\Re(t)=-\epsilon$, by \eqref{e:additive_fe_k=2} and \eqref{e:f_bchiR'}, we get
\begin{multline}\label{e:convexR_0}
(M_d^2R_d')^{\frac{t}{2}} 
\Lambda(t+\frac{1}{2},f, \frac{a}{d})
\\ \ll
(M_d^2R_d')^{\frac{1+\epsilon}{2}} 
\frac{1}{\varphi(r_d)} 
\sum_{\substack{r_{d*}\mid r_d,\\ \frac{r_d}{r_{d*}} \text{ square-free} \\ \gcd(r_d, r_d/r_{d*})=1}} 
\sum_{\substack{\chi\mod{r_{d*}} \\ \text{ primitive}}}
\sqrt{r_{d*}} 
\sum_{m=1}^\infty \frac{\left|m^{-\frac{1}{2}} b_{\chi, R_d'}(m)\right|}{m^{1+\epsilon}}. 
\end{multline}
Note that $b_{\chi, R_d'}(m)=0$ unless $Q_{d*} \nmid m$. 
Applying the bound \eqref{e:bchiR'_bound}, for any $0< \epsilon' < \epsilon$, we get
\begin{multline*}
\sum_{m=1}^\infty \frac{\left|m^{-\frac{1}{2}} b_{\chi, R_d'}(m)\right|}{m^{1+\epsilon}} 
\ll
Q_{d*}^{-\frac{1}{2}-\epsilon} 
\sigma_{-1+2\epsilon'}(r_{d*0}) r_{d*0}^{\frac{1}{2}-\epsilon'}
\sum_{n=1}^\infty \frac{n^{\epsilon'}}{n^{1+\epsilon}}
\\ 
\leq \sigma_{-1+2\epsilon'}(r_{d*0}) r_{d_{*0}}^{\frac{1}{2}-\epsilon'}\zeta(1+\epsilon-\epsilon')
\ll r_{d*0}^{\frac{1}{2}+\epsilon''}
\end{multline*}
since $Q_{d*} \in \NN$. 
Also we have
\[r_{d*0}= \prod_{\substack{p\mid r_*, p\nmid q' \\ F_\chi\mid U_p\neq 0}} p^2 
\prod_{\substack{p\mid \gcd(r_*, q') \\ F_\chi\mid U_p\neq 0}} p 
\leq \prod_{p\mid r_{d*}} p^2.\]
Applying this to \eqref{e:convexR_0}, we get
\begin{multline}
(M_d^2R_d')^{\frac{t}{2}} 
\Lambda(t+\frac{1}{2},f, \frac{a}{d})
 \ll (M_d^2R_d')^{\frac{1+\epsilon}{2}} 
\frac{1}{\varphi(r_d)} 
\sum_{\substack{r_{d*}\mid r_d,\\ \frac{r_d}{r_{d*}} \text{ square-free} \\ \gcd(r_{d*}, r_d/r_{d*})=1}} 
\sqrt{r_{d*}} \prod_{p\mid r_{d*}} p^{1+2\epsilon''}
\sum_{\substack{\chi\mod{r_{d*}} \\ \text{ primitive}}}1
\\ \leq 
(M_d^2R_d')^{\frac{1+\epsilon}{2}} 
\sum_{\substack{r_{d*}\mid r_d,\\ \frac{r_d}{r_{d*}} \text{ square-free} \\ \gcd(r_{d*}, r_d/r_{d*})=1}} 
\sqrt{r_{d*}} \prod_{p\mid r_{d*}} p^{1+2\epsilon''}, 
\end{multline}
since $r_{d*}\mid r_d$ and 
$\sum_{\substack{\chi\bmod{r_{d*}} \\ \text{ primitive }}} 1 \leq \varphi(r_{d*}) \leq \varphi(r_d).$
Upon setting $\frac{r_d}{r_{d*}}=\ell$, the right-hand side becomes 
\begin{multline*}
 (M_d^2 R_d')^{\frac{1+\epsilon}{2}} 
\sum_{\substack{\ell\mid r_d \\ \text{ square-free} \\ \gcd(r_d/\ell, \ell)=1}} 
\sqrt{r_d/\ell} \prod_{p\mid r_d/\ell} p^{1+2\epsilon''}
= (M_d^2 R_d')^{\frac{1+\epsilon}{2}} r_d^{\frac{1}{2}} \prod_{p\mid r_d} p^{1+2\epsilon''} 
\sum_{\substack{\ell\mid r_d \\ \text{ square-free} \\ \gcd(r_d/\ell, \ell)=1}} \ell^{-\frac{3}{2}-2\epsilon''}
\\ \ll (M_d^2 R_d')^{\frac{1+\epsilon}{2}} r_d^{\frac{1}{2}} \prod_{p\mid r_d} p^{1+\epsilon'''} . 
\end{multline*}
Thus 
\begin{equation}\label{convex1_1}
(M_d^2R_d')^{\frac{t}{2}} 
\Lambda(t+\frac{1}{2},f, \frac{a}{d})
\ll 
[q, d^2]^{\frac{1}{2}} (dq)^{\epsilon} r_d^{\frac{1}{2}} \prod_{p\mid r_d} p, 
\end{equation}
as $M_d^2 R_d'=[q, d^2]$. 
More explicitly, 
\[[q, d^2]^{\frac{1}{2}} r_d^{\frac{1}{2}} \prod_{p\mid r_d} p
= q^{\frac{1}{2}} d\frac{\prod_{p\mid d, \ord_p(d) < \ord_p(q)} p^{\frac{1}{2}\ord_p(d)+1}}{(q, d^2)^{\frac{1}{2}}}\]
and by examining the exponent of each $p$ in the right-hand side, we see that it is 

$$
\\ \leq q^{\frac{1}{2}}d \prod_{\substack{p\mid d \\ \ord_p(d) < \frac{1}{2}\ord_p(q)}} p^{-\frac{1}{2}\ord_p(d) +1} 
\prod_{\substack{p\mid d \\ \frac{1}{2} \ord_p(q) \leq \ord_p(d) < \ord_p(q)}} p^{-\frac{1}{2}(\ord_p(q)-\ord_p(d))+1}.
$$
When $p\mid d$ and $\ord_p(d) < \ord_p(q)$, both $\ord_p(d)\geq 1$ and $\ord_p(q)-\ord_p(d)\geq 1$.
So we get
\begin{equation}\label{e:convex1_1_cond}
\frac{\prod_{p\mid d, \ord_p(d) < \ord_p(q)} p^{\frac{1}{2}\ord_p(d)+1}}{(q, d^2)^{\frac{1}{2}}}
\leq \prod_{\substack{p\mid d \\ \ord_p(d) < \ord_p(q)}} p^{\frac{1}{2}}. 
\end{equation}

Combining \eqref{e:convex1_1_cond} with \eqref{convex1_1}, we get \eqref{convex1} for $\Re(t)=-\epsilon$. 

Recall that at $\Re(t)=1+\epsilon$, 
\[
\Lambda(t+\frac{1}{2},f, \frac{a}{d})
\ll (qd)^{\epsilon}.\]
Similarly, by \eqref{convex1}, for $\Re(t)=-\epsilon$, 
\[
\Lambda (t+\frac{1}{2}, f, \frac{a}{d})
\ll (dq)^{\epsilon} d q^{\frac{1}{2}}   \prod_{\substack{p\mid d \\ \ord_p(d) < \ord_p(q)}} p^{\frac{1}{2}}.
\]
 By the Phragm\'en-Lindel\"of convexity
principle and \eqref{e:Lambda} we deduce the proposition.

\end{proof}

Finally, the functional equation of Corollary~\ref{cFE} implies the approximate functional equation (see e.g. \cite[Theorem~5.3]{IK}).  
This states
\begin{multline}\label{approx}
L\left(1, f, \frac{a}{d}\right) 
= \sum_{n\geq 1} \frac{a(n) e^{2\pi in \frac{a}{d}}}{n}
V\left(\frac{M_d {R_d'}^{\frac{1}{2}} X}{2\pi n}\right)
\\- \frac{1}{\varphi(r_d)} 
\sum_{\substack{r_{d*}\mid r_d \\ \frac{r_d}{r_{d*}} \text{ square-free} \\ \gcd(r_{d*}, r_d/r_{d*})=1}}
 \sum_{\substack{\chi\bmod{r_{d*}}, \\ \text{ primitive }}}  
\tau(\overline{\chi}) 
\chi(a_2 \overline{(r_d/r_{d*})}) 
\chi^2(M_d)
\\ \times \sum_{n=1}^\infty 
\frac{b_{\chi, R'}\left(Q_{d*}n\right) 
e^{-2\pi i  Q_{d*} n \frac{\overline{R_d'a_1}}{M_d}}}{Q_{d*} n}
V\left(\frac{M_d {R_d'}^{\frac{1}{2}}}{2\pi Q_{d*} nX}\right)
\end{multline}
for all $X>0$, with 
\begin{equation}\label{e:Vy}
V(y):=\frac{1}{2 \pi i} \int_{(2)}(2 \pi y)^{u}G(u)\Gamma(u) du.
\end{equation}
Here
$G(u)$ is any even function which is entire and bounded in vertical strips, of arbitrary polynomial decay as $|\Im u| \to \infty$ and such that $G(0)=1.$

\section{The asymptotics of $A^{\pm}_h(M)$ as $M \to \infty$.}\label{as}
Recall the description for $\alpha_{n, M}(t)$ given in \eqref{alpha2twist}
\[
\alpha_{n, M}(t)
=\frac{1}{M}\sum_{d|M} \sum_{\substack{a \mod d \\ \gcd(a, d)=1}} 
e^{-2\pi i n\frac{a}{d}}L\big(f, t, \frac{a}{d}\big). 
\]
Our aim is to analyze the asymptotics of 
\[
A_h^\pm (M) = \sum_{n\in \mathbb{Z}} \widehat{h}(n) (\alpha_{-n, M}(1) \pm \alpha_{n, M}(1)), 
\]
as $M\to \infty$. 

We first prove the formula \eqref{alpha1} for $\alpha_{n, M}(1)$. 
To accomplish this, we first apply the approximate functional equation given in \eqref{approx}. 
Then we estimate error terms by applying Weil's bound for Kloosterman sums 
and making use of our explicitly described terms. 

For $ d\mid M$, 
we recall the notations $M_d$, $r_d$ and $R_d$ given in Corollary~\ref{cFE}:
\[d=M_dr_d, \quad \gcd(M_d, r_d)=1, \quad r_d\mid R_d, \quad R_d\mid q 
 \quad\text{ and }\quad \gcd(q/R_d, R_d)=1. \]
%
Moreover $r_d < R_d$ unless $r_d=R_d=1$. 
Also $R_d' = [R_d, r_d^2]$, $M_d^2 R_d' = [q, d^2]$. 
For any divisor $r_{d*}$ of $r_d$, we have 
$R_{d*}'\mid [R_{d*}, r_{d*}^2]$ and $r_{d*0}$ 
as described in  Corollary~\ref{cFE}. 
Recall that $Q_{d*}=\frac{R_{d*} R_d'}{R_{d*}' R_d r_{d*0}}$. 
and it is proved in Lemma~\ref{lem:decompose_f^chi} that $Q_{d*}\in \NN$. 
Finally, $a_1 \mod{M_d}$ and $a_2 \mod{r_d}$ are such that 
$a \equiv a_1 r_d+a_2 M_d\mod{d}$.

We apply \eqref{approx} to each $L(t, f, \frac{a}{d})$, with $X=X_d$, and substitute into \eqref{alpha2twist} with $t=1$:
\begin{multline}
\alpha_{n, M}(1)
=
\frac{1}{M} \sum_{d\mid M} \sum_{\substack{a\bmod{d} \\ \gcd(a, d)=1}} 
e^{-2\pi in \frac{a}{d}} 
\sum_{\ell\geq 1} \frac{a(\ell) e^{2\pi i\ell \frac{a}{d}}}{\ell}
V\left(\frac{M_d {R_d'}^{\frac{1}{2}} X_d}{2\pi \ell}\right)
\\ - \frac{1}{M} \sum_{d\mid M} \sum_{\substack{a\bmod{d} \\ \gcd(a, d)=1}} 
e^{-2\pi in \frac{a}{d}} 
\frac{1}{\varphi(r_d)} 
\sum_{\substack{r_{d*}\mid r_d \\ \frac{r_d}{r_{d*}} \text{ square-free} \\ \gcd(r_{d*}, r_d/r_{d*})=1}}
 \sum_{\substack{\chi\bmod{r_{d*}}, \\ \text{ primitive }}}  
\tau(\overline{\chi}) 
\chi(a_2 \overline{(r_d/r_{d*})}) 
\chi^2(M_d)
\\ \times 
\sum_{\ell=1}^\infty 
\frac{b_{\chi, R'}\left(Q_{d*} \ell\right) 
e^{-2\pi iQ_{d*}\ell \frac{\overline{R_d'a_1}}{M_d}}}
{Q_{d*}\ell}
V\left(\frac{M_d {R_d'}^{\frac{1}{2}}}
{2\pi Q_{d*}\ell X_d}\right). 
\end{multline}
Here $\overline{R_d' a_1} R_d'a_1\equiv 1\mod{M_d}$. 
Set $X_d= \frac{X}{M_d R_d'^{\frac{1}{2}}}$, with $X$ independent of $d$.
Since
\[\sum_{d\mid M} \sum_{\substack{a\mod{d} \\ \gcd(a, d)=1}} e^{2\pi i\frac{a}{d}(-n+\ell)} 
=
\begin{cases}
M, & \text{ if } n\equiv \ell\mod{M} \\
0, & \text{ otherwise}
\end{cases}\]
and
\begin{multline}
\sum_{\substack{a\bmod{d} \\ \gcd(a, d)=1}} 
e^{-2\pi in \frac{a}{d}} 
\chi(a_2\overline{(r_d/r_{d*})}) e^{-2\pi i Q_{d*} \ell \frac{\overline{R_d' a_1}}{M_d}} 
\\ = 
\sum_{\substack{a_1\bmod{M_d} \\ \gcd(a_1, M_d)=1}} 
e^{-2\pi in \frac{a_1}{M_d}} 
e^{-2\pi i Q_{d*} \ell \frac{\overline{R_d' a_1}}{M_d}} 
\sum_{\substack{a_2\bmod{r_d} \\ \gcd(a_2, r_d)=1}} 
\chi(a_2\overline{(r_d/r_{d*})}) e^{-2\pi in \frac{a_2}{r_d}} 
\\ = S\left(n, \ell Q_{d*} \overline{R_d'}; M_d\right)
\chi(-\overline{r_d/r_{d*}}) c_{\tilde \chi_{r_d}}(n),
\end{multline}
we get 
\begin{multline}\label{a(1)}
\alpha_{n, M}(1)
=
\sum_{\substack{\ell\geq 1 \\ \ell \equiv n \mod{M}}} \frac{a(\ell)}{\ell}
V\left(\frac{X}{2\pi \ell}\right)
\\ - \frac{1}{M} \sum_{d\mid M} \frac{1}{\varphi(r_d)} 
\sum_{\substack{r_{d*}\mid r_d \\ \frac{r_d}{r_{d*}} \text{ square-free} \\ \gcd(r_{d*}, r_d/r_{d*})=1}}
 \sum_{\substack{\chi\bmod{r_{d*}}, \\ \text{ primitive }}}  
\tau(\overline{\chi}) 
\chi(-\overline{(r_d/r_{d*})}M_d^2)
c_{\tilde \chi_{r_d}}(n)
\\ \times 
\sum_{\ell=1}^\infty 
\frac{b_{\chi, R'}\left(Q_{d*}\ell\right)}
{Q_{d*}\ell}
S(n, \ell Q_{d*} \overline{R_d'}; M_d)
V\left(\frac{M_d^2 R_d'}
{2\pi Q_{d*} \ell X}\right). 
\end{multline}
Here $\tilde \chi_{r_d}$ is a Dirichlet character modulo $r_d$, which is induced from the primitive character $\chi\mod{r_{d*}}$. 

For the last sum of \eqref{a(1)} we use Weil's bound for Kloosterman sums, which implies, as $\gcd(R_d', M_d)=1$ and $\gcd(Q_{d*}, M_d)=1$, 
\begin{equation}\label{e:Weil_bound}
\left|S(n, \ell Q_{d*} \overline{R_d'}; M_d)\right|
\leq \gcd(n, \ell, M_d)^{\frac{1}{2}} M_d^{\frac{1}{2}} \sigma_0(M_d)
\end{equation}
%
By applying \eqref{e:Weil_bound} and \eqref{e:bchiR'_upper}, we get
\begin{multline}\label{e:b_Kloosterman_upper1}
\left|\sum_{\ell=1}^\infty 
\frac{b_{\chi, R'}\left(Q_{d*}\ell\right)}{Q_{d*}\ell}
S(n, \ell Q_{d*} \overline{R_d'}; M_d)
V\left(\frac{M_d^2 R_d'}
{2\pi Q_{d*}\ell X}\right)\right|
\\ \ll_{\epsilon} M_d^{\frac{1}{2}} \sigma_0(M_d) 
r_{d*0}^{\frac{1}{2}-\epsilon} \sigma_{-1+2\epsilon}(r_{d*0})
\sum_{\ell=1}^\infty 
\frac{\gcd(n, \ell, M_d)^{\frac{1}{2}}}{\ell^{\frac{1}{2}-\epsilon}}
V\left(\frac{M_d^2 R_d'}
{2\pi Q_{d*}\ell X}\right)
\end{multline}
for any $\epsilon>0$. 
Since $\gcd(n, \ell, M_d)\mid \gcd(n, M_d)$, 
we get 
\[\leq M_d^{\frac{1}{2}} \sigma_0(M_d) 
r_{d*0}^{\frac{1}{2}-\epsilon} \sigma_{-1+2\epsilon}(r_{d*0})
\sum_{g\mid \gcd(n, M_d)} g^{\epsilon} 
\sum_{\ell=1}^\infty \frac{1}{\ell^{\frac{1}{2}-\epsilon}}
V\left(\frac{M_d^2 R_d'}{2\pi Q_{d*}\ell g X}\right).\]
If $y< M^{-\epsilon'}$, one easily checks, by moving the line of integration in \eqref{e:Vy} to the right, 
that $V(y) \ll M^{-K_{\epsilon'}}$, for arbitrarily large $K_\epsilon$. 
Consequently, 
\begin{multline}
\sum_{g\mid \gcd(n, M_d)} g^{\epsilon} 
\sum_{\ell=1}^\infty \frac{1}{\ell^{\frac{1}{2}-\epsilon}}
V\left(\frac{M_d^2 R_d'}{2\pi Q_{d*}\ell g X}\right)
\ll \sum_{g\mid \gcd(n, M_d)} g^{\epsilon} 
\sum_{\ell\ll M^{\epsilon'} \frac{M_d^2 R_d'}{2\pi Q_{d*} g X}}  \frac{1}{\ell^{\frac{1}{2}-\epsilon}}
\\ \ll \sum_{g\mid \gcd(n, M_d)} g^{\epsilon} M^{\epsilon'} 
\left( \frac{M_d^2 R_d'}{2\pi Q_{d*} g X}\right)^{\frac{1}{2}+\epsilon}
\leq \sum_{g\mid \gcd(n, M_d)} g^{\epsilon} M^{\epsilon'} 
\left( \frac{M_d^2 R_d'}{2\pi g X}\right)^{\frac{1}{2}+\epsilon}
\end{multline}
since $Q_{d*} \geq 1$ (See Proposition~\ref{prop:Fourier_upper}).
Applying this to \eqref{e:b_Kloosterman_upper1}, we get
\begin{multline}\label{e:b_Kloosterman_upper2}
\left|\sum_{\ell=1}^\infty 
\frac{b_{\chi, R'}\left(Q_{d*}\ell\right)}{Q_{d*}\ell}
S(n, \ell Q_{d*} \overline{R_d'}; M_d)
V\left(\frac{M_d^2 R_d'}
{2\pi Q_{d*}\ell X}\right)\right|
\\ \ll_{\epsilon} M_d^{\frac{1}{2}} \sigma_0(M_d) 
r_{d*0}^{\frac{1}{2}-\epsilon} \sigma_{-1+2\epsilon}(r_{d*0})
\sum_{g\mid \gcd(n, M_d)} g^{\epsilon} M^{\epsilon'} 
\left( \frac{M_d^2 R_d'}{2\pi g X}\right)^{\frac{1}{2}+\epsilon}
\\ \ll X^{-\frac{1}{2}-\epsilon} M^{\epsilon'}
M_d^{\frac{3}{2}+2\epsilon} \sigma_0(M_d) 
{R_d'}^{\frac{1}{2}+\epsilon}
\prod_{p\mid r_{d*}} p
\end{multline}
By definition of $R_d'$, we have
\[R_d' = [R_d, r_d^2] 
= R_d r_d\frac{r_d}{(R_d, r_d^2)} 
= R_d r_d \frac{1}{(R_d/r_d, r_d)} \leq R_dr_d.\]
 The third equality holds since $r_d\mid R_d$. 
Also recall that $d=M_dr_d$. 
Therefore, referring to \eqref{e:b_Kloosterman_upper2}, we have 
\begin{multline}\label{e:b_Kloosterman_sum_bound}
\left|\sum_{\ell=1}^\infty 
\frac{b_{\chi, R'}\left(Q_{d*}\ell\right)}{Q_{d*}\ell}
S(n, \ell Q_{d*} \overline{R_d'}; M_d)
V\left(\frac{M_d^2 R_d'}
{2\pi Q_{d*}\ell X}\right)\right|
\\ \ll X^{-\frac{1}{2}-\epsilon} M^{\epsilon'}
d^{\frac{3}{2}+2\epsilon} \sigma_0(M_d) 
r_d^{-1-\epsilon} 
R_d^{\frac{1}{2}+\epsilon}
\prod_{p\mid r_{d*}} p.
\end{multline}

Note (see \eqref{e:c_chi_nonzero}) that $c_{\tilde \chi_{r_d}}(n)=0$ 
if $\frac{r_d}{r_{d*}r_{d0}}\nmid n$. 
Here $r_{d0} = \prod_{p\mid r_d, p\nmid r_{d*}} p = \frac{r_d}{r_{d*}}$ since $r_d/r_{d*}$ is square-free and 
$\gcd(r_{d*}, r_d/r_{d*})=1$. 
So $\frac{r_d}{r_{d*}r_{d0}}=1$ and we get
\[c_{\tilde \chi_{r_d}}\left(n\right)
= \tau(\chi) \overline{\chi}(n)\mu\left(\gcd(r_d/r_{d*},n)\right) \varphi\left(\gcd(r_d/r_{d*}, n)\right).\]
Thus
\begin{equation}\label{e:Ram_bound}
\left|c_{\tilde \chi_{r_d}}(n)\right|
\leq \sqrt{r_{d*}} \varphi(\gcd(n, r_d/r_{d*})) 
\leq \frac{r_d}{\sqrt{r_{d*}}}. 
\end{equation}
By applying \eqref{e:b_Kloosterman_sum_bound} and \eqref{e:Ram_bound} 
to the second summation in \eqref{a(1)}, 
we get
\begin{multline}
\frac{1}{M} \sum_{d\mid M} \frac{1}{\varphi(r_d)} 
\sum_{\substack{r_{d*}\mid r_d \\ \frac{r_d}{r_{d*}} \text{ square-free} \\ \gcd(r_{d*}, r_d/r_{d*})=1}}
\sum_{\substack{\chi\bmod{r_{d*}} \\ \text{ primitive}}} 
\tau(\overline{\chi}) 
\chi(-\overline{(r_d/r_{d*})}M_d^2)
c_{\tilde \chi_{r_d}}(n)
\\ \times 
\sum_{\ell=1}^\infty 
\frac{b_{\chi, R'}\left(Q_{d*}\ell\right)}
{Q_{d*}\ell}
S(n, \ell Q_{d*} \overline{R_d'}; M_d)
V\left(\frac{M_d^2 R_d'}
{2\pi Q_{d*} \ell X}\right)
\\ \ll 
\frac{X^{-\frac{1}{2}-\epsilon} M^{\epsilon'}}{M} \sum_{d\mid M}
\sigma_0(M_d) d^{\frac{3}{2}+2\epsilon} r_d^{-\epsilon} R_d^{\frac{1}{2}+\epsilon}
\sum_{\substack{\ell\mid r_d,\\ \text{ square-free} \\ \gcd(r_d/\ell, \ell)=1}}
\prod_{p\mid r_{d}/\ell} p
\\ \ll \frac{X^{-\frac{1}{2}-\epsilon} M^{\epsilon'}}{M} 
\sum_{d\mid M} \sigma_0(M_d) d^{\frac{3}{2}+2\epsilon} r_d^{-\epsilon} R_d^{\frac{1}{2}+\epsilon} \prod_{p\mid r_d} p. 
\end{multline}
 
Since $d\mid M$, we see that the right-hand side is 
\[\leq \frac{M^{\frac{1}{2}+2\epsilon+\epsilon'}\sigma_0(M)}{X^{\frac{1}{2}+\epsilon}} 
\sum_{d\mid M} 
\prod_{\substack{p\mid q,\\ p\nmid d}} p^{\left(\frac{1}{2}+\epsilon\right)\ord_p(q)} 
\prod_{p\mid r_d} p^{1+\left(\frac{1}{2}+\epsilon\right) \ord_p(q)}.\]
For each $d\mid M$, 
\begin{multline}
\prod_{\substack{p\mid q,\\ p\nmid d}} p^{\left(\frac{1}{2}+\epsilon\right)\ord_p(q)} 
\prod_{p\mid r_d} p^{1+\left(\frac{1}{2}+\epsilon\right) \ord_p(q)}
= q^{\frac{1}{2}+\epsilon}
\prod_{p\mid \gcd(d, q)} p^{-\left(\frac{1}{2}+\epsilon\right) \ord_p(q)} 
\prod_{p\mid r_d} p^{1+\left(\frac{1}{2}+\epsilon\right) \ord_p(q)}
\\ \leq q^{\frac{1}{2}+\epsilon} \prod_{\substack{p\mid d \\ \ord_p(d) < \ord_p(q)}} p
\leq q^{\frac{1}{2}+\epsilon} \prod_{\substack{p\mid \gcd(q, M) \\ p^2\mid q}} p. 
\end{multline}

The condition of the last product implies that the term is $1$ when $q$ is square-free or 
$\gcd(q, M)=1$. 
We thus have
\begin{equation}\label{alpha1}
\alpha_{n, M}(1)
=
\sum_{\ell \equiv n \bmod M} \frac{a(\ell)}{\ell}
V\left(\frac{ X}{2\pi \ell}\right)  
\\  + \mathcal{O}\bigg(M^{\frac{1}{2}+2\epsilon+\epsilon'} \sigma_0(M)^2 X^{-\frac{1}{2}-\epsilon} 
q^{\frac{1}{2}+\epsilon} \prod_{\substack{p\mid \gcd(q, M) \\ p^2\mid q}} p\bigg). 
\end{equation}

This allows us to prove
\begin{lem} \label{bounda} 
Let $M>1.$ For each $X>0$, we have
\begin{equation}\label{assem}
\sum_{n \in \mathbb Z } \hat h(n) \alpha_{\pm n, M}(1)
= \sum_{n \in \mathbb Z} \hat h(n) 
\sum_{m \equiv \pm n \mod M} \frac{a(m)}{m}V\left(\frac{ X}{2\pi m}\right) 
+ \mathcal{O}\bigg(X^{-\frac{1}{2}-\epsilon} M^{\frac{1}{2}+\epsilon} 
 q^{\frac{1}{2}+\epsilon} \prod_{\substack{p\mid \gcd(q, M) \\ p^2\mid q}} p 
\bigg), 
\end{equation}
for any $\epsilon>0$.
\end{lem}
\begin{proof}
Replacing $\ell$ by $m$ in \eqref{alpha1}, we get
\begin{multline}\label{tempasse}
\sum_{n \in \mathbb Z } \hat h(n) \alpha_{\pm n, M}(1)
= \sum_{n \in \mathbb Z} \hat h(n) 
\sum_{m \equiv \pm n \mod M} \frac{a(m)}{m}V\left(\frac{ X}{2\pi m}\right)
\\  + \mathcal{O}\bigg(M^{\frac{1}{2}+2\epsilon+\epsilon'} \sigma_0(M)^2 X^{-\frac{1}{2}-\epsilon} 
q^{\frac{1}{2}+\epsilon} \prod_{\substack{p\mid \gcd(q, M) \\ p^2\mid q}} p
\bigg(\sum_{n\in \ZZ}| \widehat{h}(n)|\bigg) \bigg). 
\end{multline}
Now, for $h=h_{\delta}$ with $\delta = \delta_M> M^{-1+\eta}$,  
for some $0< \eta< 1$,  \eqref{unifbou} implies that 
\[\sum_{\substack{n \in \mathbb{Z} \\ \delta (|n|+1)> (|n|+1)^{1-\eta} M^{1-\eta} }} |\hat h_{\delta}(n)| \ll_K 
\sum_{\substack{n \in \mathbb{Z} \\ \delta (|n|+1)> (|n|+1)^{1-\eta}  M^{1-\eta} }}
(1 + |n|)^{-1}  (\delta(|n|+1))^{-K},\]
for arbitrary $K$.    
Choosing $K = K'/(1-\eta)$, with $K' \gg 1$,
we see that this portion of the sum is $\ll M^{-K'}$, for arbitrary $K'$.

Taking the remaining portion of the sum,
$$ \sum_{\substack{n \in \mathbb{Z} \\ \delta (|n|+1)\le (|n|+1)^{1-\eta} M^{1-\eta} }} 
|\hat h_{\delta}(n)|
\ll \sum_{\substack{n \in \mathbb{Z} \\ |n| \ll M^{2/\eta -2}}} 
|\hat h_{\delta}(n)| \ll M^{\epsilon''}, $$
as for $n \ne 0$, $|\hat h_{\delta}(n)| \le 1/|n|$.
Thus the error term of \eqref{tempasse} is
\[\mathcal{O}\bigg(X^{-\frac{1}{2}-\epsilon} M^{\frac{1}{2}+\epsilon} 
 q^{\frac{1}{2}+\epsilon} \prod_{\substack{p\mid \gcd(q, M) \\ p^2\mid q}} p \bigg),\]
with a new $\epsilon>0$.
\end{proof}

The next proposition gives us an estimate for the first term of the right-hand side of \eqref{assem}
\begin{lem}\label{asym}
For $h=h_{\delta}$  with $\delta=\delta_M> M^{-1+\eta}$, for some fixed $\eta>0$, we have, 
\[\sum_{\substack{n \in  \mathbb Z  }} \hat h_{\delta} (n)
\sum_{m \equiv \pm n \mod M} \frac{a(m)}{m}V\left(\frac{ X}{2\pi m}\right)
=\sum_{\substack{n \ge 1 }} \hat h(\pm n) \frac{a(n)}{n}
+\mathcal{O}\left(X^{-\frac12 + \epsilon}\right)
+ \mathcal{O}\left ( X^{\frac12+\epsilon}M^{-1+\epsilon} \right ).\]
\end{lem}
\begin{proof} 
Referring to \eqref{assem}, we consider
\begin{multline*}
\sum_{n \in \mathbb Z} \hat h(n) 
\sum_{m \equiv \pm n \mod M} \frac{a(m)}{m}V\left(\frac{ X}{2\pi m}\right) 
\\ =
\sum_{n\ge 1} \hat h(\pm n) 
\frac{a(n)}{n} V\left(\frac{X}{2\pi n}\right)
+ 
\sum_{n\in \ZZ} \sum_{\substack{m\equiv\pm n\mod{M} \\ m\neq \pm n}}
\frac{a(m)}{m}V\left(\frac{ X}{2\pi m}\right).
\end{multline*}

We first consider the  diagonal  terms with $m =\pm n$ in the sum. 
Upon moving the line of integration  of the integral 
\begin{equation*}
V(y)=\frac{1}{2 \pi i} \int_{(2)}(2 \pi y)^{u}G(u)\Gamma(u) du,
\end{equation*}
(see \eqref{e:Vy})  
we get
\begin{multline*}
\sum_{n\ge 1} \hat h(\pm n) 
\frac{a(n)}{n} V\left(\frac{X}{2\pi n}\right)
\\ =
\sum_{\substack{n \ge 1 }} \hat h(\pm n) \frac{a(n)}{n}+
\sum_{\substack{n \ge 1 }} \hat h (\pm n) \frac{a(n)}{n}\frac{1}{2 \pi i}\int\limits_{(-\frac12+\epsilon)}\left ( \frac{X}{2 \pi n}\right )^{u}G(u)\frac{\Gamma(u+1) }{(2 \pi )^u}\frac{du}{u}.
\end{multline*}
Since the second sum is $\ll \sum_{n \ge 1}| \hat h (\pm n)| n^{-\epsilon'} X^{-1/2 + \epsilon}$, inequality \eqref{unifbou} implies that the sum converges and we have
\begin{equation}\label{I1}
\sum_{\substack{n \ge 1 }} \hat h (\pm n) \frac{a(n)}{n}+\mathcal{O}\left(X^{-1/2 + \epsilon}\right).
\end{equation} 

Now we are left with the  off-diagonal,  $n \neq \pm m$ terms. 
Note that the length of the sum over $m$ is $X^{1+\epsilon}$ by the rapid decay of $V$. 
We separate into two cases: $|n|\leq M/2$ and $|n|>M/2$. 

For the latter, \eqref{unifbou} implies 
\begin{multline*}
\sum_{\substack{n \in  \mathbb Z  \\ |n|>M/2}} \hat h (n)
\sum_{\substack{m \equiv \pm n \mod M \\m \ne \pm n}} 
\frac{a(m)}{m}V\left(\frac{X}{2 \pi m}\right) 
\ll \sum_{\substack{n \in  \mathbb Z   \\ |n|>M/2}} \frac{(\delta_M(1 +|n|))^{-K}}{1+|n|} 
\sum_{\substack{m \equiv \pm n \mod M\\ 0< m \ll X^{1+\epsilon}, m \ne \pm n}} \frac{1}{m^{1/2-\epsilon}} 
\\ \ll 
\delta_M^{-K}  \sum_{\substack{n \in  \mathbb Z   \\ |n|>M/2}} (1 +|n|)^{-K-1}
\sum_{\substack{0< m \ll X^{1+\epsilon}}} \frac{1}{m^{1/2-\epsilon}}
\ll \delta_M^{-K} M^{-K} X^{1/2+\epsilon}.
\end{multline*}
Since, $\delta_M > 1/M^{1-\eta}$, that is, $\delta_M M > M^\eta$, 
by 
 assuming  $X$ will be less than some fixed power of  $Mq$, 
we get $\mathcal{O}\left((qM)^{-K}\right)$ with $K>1$ arbitrarily large.

For the former case we note that as $m \ne \pm n$,  
the congruence relation modulo $M$ forces $m>M/2$. 
We then calculate using $K = 0$ in \eqref{unifbou}, 
and recalling that the contribution from $m> X^{1 + \epsilon}$ 
is smaller than $(Mq)^{-K}$ for arbitrary $K \gg 1$, we get
\begin{multline}\label{e:assem_1term}
\sum_{\substack{n \in  \mathbb Z,  \\ |n| \leq M/2}} \hat h (n)
\sum_{\substack{m \equiv \pm n \mod M \\ m \ne \pm n}} 
\frac{a(m)}{m}V\left(\frac{X}{2 \pi m}\right)
\ll \sum_{\substack{|n|\leq M/2}}\frac{1}{|n|+1} 
\sum_{0<|\ell| \ll X^{1+\epsilon}/M} \frac{a(\pm n + M \ell)}{(\pm n + M \ell)} 
\\ \ll 
\sum_{|n| \leq M/2} \frac{1}{|n|+1}  
\sum_{\substack{0<|\ell| \ll X^{1+\epsilon}/M \\ (\pm n + M\ell)>0}} 
\frac{ M^{-1/2+\epsilon}}{(\frac{\pm n}{M} + \ell)^{1/2-\epsilon}} 
\ll  X^{1/2+\epsilon}M^{-1+\epsilon}.
\end{multline}
Combining \eqref{I1} with \eqref{e:assem_1term} yields the proposition. 
\end{proof}

We can combine Lemma \ref{asym} with Lemma \ref{bounda} to get the asymptotics of
$\sum \hat h(n) \alpha_{\pm n, M}(1)$. 
To this end, we will compare the error terms produced in Lemmas \ref{bounda}
and \ref{asym} to determine a value of $X$ that gives the optimal bound.
Setting the error terms from \eqref{assem} and \eqref{e:assem_1term} equal,
we get
\[X^{-\frac{1}{2}} M^{\frac{1}{2}} 
 q^{\frac{1}{2}} \prod_{\substack{p\mid \gcd(q, M) \\ p^2\mid q}} p 
= X^{\frac{1}{2}} M^{-1}.\]
This gives us 

\[X=M^{\frac{3}{2}}q^{\frac{1}{2}} \prod_{\substack{p\mid \gcd(q, M) \\ p^2\mid q}} p.\]
Thus the error from these two contribution is 
\[X^{\frac{1}{2}} M^{-1} (Mq)^{\epsilon}
= (Mq)^{\epsilon} M^{-\frac{1}{4}} q^{\frac{1}{4}} \prod_{\substack{p\mid \gcd(q, M) \\ p^2\mid q}} p^{\frac{1}{2}}  \]

The remaining error (from \eqref{I1}) is dominated by these terms since $X^{-1/2 + \epsilon} \ll X^{-1/2 + \epsilon} M$. 
From this, together with Lemmas \ref{bounda}, \ref{asym}, we deduce the following 
\begin{prop}\label{asymFIN}
Let $M>1.$ For $h=h_{\delta}$  with $\delta=\delta_M> M^{-1+\eta}$ for some fixed $0<\eta<1$, we have, 
\[\sum_{n \in \mathbb Z } \hat h_{\delta}(n) \alpha_{\pm n, M}(1)
=\sum_{\substack{n \ge 1 }} \hat h_{\delta}(\pm n) \frac{a(n)}{n}
+\mathcal{O}\bigg( (Mq)^{\epsilon} M^{-\frac{1}{4}} q^{\frac{1}{4}} \prod_{\substack{p\mid \gcd(q, M) \\ p^2\mid q}} p^{\frac{1}{2}}  \bigg).\]
\end{prop}

Therefore we obtain 
\[
A_h^\pm (M) = 
\sum_{n\geq 1} (\widehat{h}_{\delta}(-n)\pm \widehat{h}_{\delta}(n)) \frac{a(n)}{n} 
+  \mathcal{O}\bigg((Mq)^{\epsilon} M^{-\frac{1}{4}} q^{\frac{1}{4}} \prod_{\substack{p\mid \gcd(q, M) \\ p^2\mid q}} p^{\frac{1}{2}} \bigg). 
\]
In the next section we study the sum 
\[\sum_{n\geq 1} (\widehat{h}_{\delta}(-n)\pm \widehat{h}_{\delta}(n)) \frac{a(n)}{n}, \]
choose $\delta_M$ for the final error term
and conclude the proof of Theorem~\ref{MAIN}.

\section{Proof of Theorem \ref{MAIN}}\label{PfMAIN}
For fixed $x$ we consider $h=h_{\delta}$. 
Combining \eqref{e:Ahpm} and Proposition~\ref{asymFIN} we deduce 
\begin{multline}\label{temp}
\frac{1}{2} A_h^\pm (M)
=\frac{1}{M}\sum_{0 \le a \le M}\left<\frac{a}{M}\right>^{\pm} h_{\delta}( \frac{a}{M})
\\ = \sum_{n\geq 1} (\widehat{h}_{\delta}(-n)\pm \widehat{h}_{\delta}(n)) \frac{a(n)}{n} 
+  \mathcal{O}\bigg((Mq)^{\epsilon} M^{-\frac{1}{4}} q^{\frac{1}{4}} \prod_{\substack{p\mid \gcd(q, M) \\ p^2\mid q}} p^{\frac{1}{2}} \bigg). 
\end{multline}

\begin{lem}\label{finboulem} 
For $h=h_{\delta}$ with $\delta=\delta_M$, 
we have
\[\sum_{n \ge 1}\hat h_{\delta}(n)  \frac{a(n)}{n}
=
\sum_{n \ge 1}\frac{1-e^{-2 \pi i n x}}{2 \pi i n}\frac{a(n)}{n}
+\mathcal{O}\left(\delta_M^{\frac12-\epsilon}\right).\]
\end{lem}
\begin{proof} 
We have
\begin{multline}\label{finbou}
\left | \sum_{n \ge 1}\hat h_{\delta}(n)  \frac{a(n)}{n}
- \sum_{n \ge 1}\frac{1-e^{-2 \pi i n x}}{2 \pi i n}\frac{a(n)}{n} \right | 
\le 
\left | \sum_{n > \delta_M^{-1}}\hat h_{\delta}(n)  \frac{a(n)}{n}\right | 
+ \left | \sum_{n > \delta_M^{-1}} \frac{1-e^{-2 \pi i n x}}{2 \pi i n}\frac{a(n)}{n} \right | 
\\+ 
\sum_{n =1}^{\delta_M^{-1}} \left | \hat h_{\delta}(n)- \frac{1-e^{-2 \pi i n x}}{2 \pi i n}\right |\frac{a(n)}{n}.   
\end{multline}

Because of \eqref{unifbou}, we have
\[\left | \sum_{n > \delta_M^{-1}}\hat h_{\delta}(n)  \frac{a(n)}{n} \right | 
\ll \sum_{n > \delta_M^{-1}}\frac{1}{n^{\frac32-\epsilon_1}} 
\ll_{\epsilon} \delta_M^{\frac12-\epsilon}.\]
Since $\frac{1-e^{-2 \pi i n x}}{2 \pi i n}$ is likewise $\ll n^{-1}$,  
the same bound holds for the second sum in the right-hand side of \eqref{finbou}.

For the last sum of \eqref{finbou}, we observe that, because of \eqref{Fourfor}, 
we have
\begin{multline}\label{Bh}
\hat h_{\delta}(n)
=\frac{1-e^{-2 \pi i n x}}{2 \pi i n}
+\frac{1}{2 \pi i n}\int_{-\frac12}^{\frac12}\phi(t) 
\left ((e^{2 \pi i \delta_M n (1-t) }-1)-e^{-2 \pi in x}(e^{-2 \pi i n \delta_M (1+t) }-1) \right )\; dt
\\ =
\frac{1-e^{-2 \pi i n x}}{2 \pi i n}+\frac{1}{2 \pi i n}\mathcal{O}\left(n\delta_M\right)
\end{multline}
because $e^{2 \pi i \delta_M n (1-t) }=1+\mathcal{O}\left(n\delta_M\right)$ 
(since $n\delta_M \le 1$).   Thus
\[\sum_{n =1}^{\delta_M^{-1}} 
\left | \hat h_{\delta}(n)- \frac{1-e^{-2 \pi i n x}}{2 \pi i n}\right |\frac{a(n)}{n}  
\ll \delta_M \sum_{n =1}^{\delta_M^{-1}} \frac{1}{n^{\frac12-\epsilon}}
\ll \delta_M  \delta_M^{-\frac12-\epsilon}=\delta_M^{\frac12-\epsilon}.\]
\end{proof}

Similarly, we can prove that 
\[\sum_{n \ge 1}\hat h_{\delta}(-n)  \frac{a(n)}{n}
=\sum_{n \ge 1}\frac{1-e^{2 \pi i n x}}{-2 \pi i n}\frac{a(n)}{n}
+\mathcal{O}(\delta_M^{\frac12-\epsilon}).\]
Entering this and Lemma \ref{finboulem} into \eqref{temp}, we derive the main terms of Theorem \ref{MAIN}. 

To determine the error term we note that the error terms we have obtained 
from our analysis are $ \delta_M^{\frac{1}{2}-\epsilon}$ from the above, 
\[\delta_M M^{\frac{1}{2}} q^{\frac{1}{4}} (qM)^{\epsilon} 
 \prod_{\substack{p\mid M \\ \ord_p(M) < \ord_p(q)}} p^{\frac{1}{4}} 
\leq \delta_M M^{\frac{1}{2}} q^{\frac{1}{4}} (qM)^{\epsilon} 
\prod_{\substack{p\mid \gcd(q, M) \\ p^2\mid q}} p^{\frac{1}{4}}  
\]
from Lemma~\ref{weights}, 
and 
\[ (Mq)^{\epsilon} M^{-\frac{1}{4}} q^{\frac{1}{4}} \prod_{\substack{p\mid \gcd(q, M) \\ p^2\mid q}} p^{\frac{1}{2}} \]
from Proposition~\ref{asymFIN}. 

Setting 

\[\delta_M M^{\frac{1}{2}} q^{\frac{1}{4}} (qM)^{\epsilon} 
\prod_{\substack{p\mid \gcd(q, M) \\ p^2\mid q}} p^{\frac{1}{4}} 
= (Mq)^{\epsilon} M^{-\frac{1}{4}} q^{\frac{1}{4}} \prod_{\substack{p\mid \gcd(q, M) \\ p^2\mid q}} p^{\frac{1}{2}} \]
gives us
\[\delta_M^{\frac{1}{2}} 
= (Mq)^{\epsilon} M^{-\frac{3}{8}} \prod_{\substack{p\mid \gcd(q, M),\\ p^2\mid q}} p^{\frac{1}{8}}\]
So certainly $\delta_{M}^{\frac{1}{2}}$ is smaller than other error terms in \eqref{temp}: 
\[\delta_M^{\frac{1}{2}} < M^{-\frac{1}{4}} q^{\frac{1}{4}} \prod_{\substack{p\mid \gcd(q, M) \\ p^2\mid q}} p^{\frac{1}{2}}.\]
Thus the final error is 
\[(Mq)^{\epsilon} M^{-\frac{1}{4}} q^{\frac{1}{4}} \prod_{\substack{p\mid \gcd(q, M) \\ p^2\mid q}} p^{\frac{1}{2}}.\]

This completes the proof of Theorem~\ref{MAIN}.

\thispagestyle{empty}
{\footnotesize
\nocite{*}
\bibliographystyle{plain}
\bibliography{DHKL_reference}
}

\end{document}